\documentclass[12pt,a4paper,twoside,reqno]{amsart}

\usepackage{tikz}
\usepackage{slashed}
\usepackage{amsmath,amscd}
\usepackage{amssymb,amsfonts,mathrsfs}
\usepackage[driver=pdftex,margin=3cm,heightrounded=true,centering]{geometry}
\usepackage{JK}
\usepackage[colorlinks=true,linkcolor=blue,citecolor=blue]{hyperref}

\newtheorem{thm}{Theorem}[section]

\usepackage{graphicx}

\newcommand{\vertiii}[1]{{\left\vert\kern-0.25ex\left\vert\kern-0.25ex\left\vert #1 
    \right\vert\kern-0.25ex\right\vert\kern-0.25ex\right\vert}}

\usepackage{amscd}

\begin{document}
\title[External products of spectral metric spaces]{External products of spectral metric spaces}

\dedicatory{Dedicated to Fritz Gesztesy on the occasion of his 70th birthday.}

\author{Jens Kaad}

\address{Department of Mathematics and Computer Science,
The University of Southern Denmark,
Campusvej 55, DK-5230 Odense M,
Denmark}

\email{kaad@imada.sdu.dk}





\begin{abstract}
In this paper, we present a characterization of compact quantum metric spaces in terms of finite dimensional approximations. This characterization naturally leads to the introduction of a matrix analogue of a compact quantum metric space. As an application, we show that matrix compact quantum metric spaces are stable under minimal tensor products and more specifically that matrix spectral metric spaces are stable under the external product operation on unital spectral triples. We present several noncommutative examples of matrix compact quantum metric spaces. 
\end{abstract}

\maketitle
\tableofcontents

\section{Introduction}
The theory of compact quantum metric spaces was initiated by Rieffel in \cite{Rie:MSA,Rie:GHQ} and expands the classical theory of compact metric spaces to the quantized (noncommutative) setting. The main object of study is the state space for an underlying system of observables which in Rieffel's approach is represented by an order unit space. The focus lies on the interplay between seminorms on the system of observables and metrics on the state space and in particular on seminorms which recover the weak-$*$-topology via their associated state space metric.

From the very beginning, there has been a substantial overlap between noncommutative geometry \'a la Connes, \cite{Con:NCG,Con:CFH,Con:GCM}, and the theory of compact quantum metric spaces. The central object in noncommutative geometry is a spectral triple and the metric aspect of noncommutative geometry relies precisely on the fact that the abstract Dirac operator induces an extended metric on the state space, \cite{Con:NCG,Con:CFH}. The relevant seminorm is here given by taking the norm of the commutator between the Dirac operator and an element in the underlying coordinate algebra. A unital spectral triple which, in this fashion, recovers the weak-$*$-topology on the state space of the coordinate algebra is called a spectral metric space. 

There has been a rich and still ongoing development with regards to the Gromov-Hausdorff convergence of compact quantum metric spaces. This development was initiated by Rieffel in \cite{Rie:GHQ} but has been significantly advanced by Latr\'emoli\`ere by taking into account more and more structure, \cite{Lat:QGP,Lat:MGP,Lat:CGP}, recently culminating in a convergence theory for spectral metric spaces, \cite{Lat:PMS}.

In this paper we advance the theory of compact quantum metric spaces by investigating the compatibility between compact quantum metric spaces and tensor products. This relationship is particularly interesting within the context of spectral metric spaces since there is a canonical tensor product construction available for spectral triples called the external product, \cite{BaJu:TBK}. This external product of spectral triples is moreover compatible with the external product in analytic $K$-homology via the Baaj-Julg bounded transform, \cite{BaJu:TBK,Kuc:PUM,HiRo:AKH,Kas:OFE}.

We may thus formulate the following question:
\medskip

\emph{Given two spectral metric spaces, can we conclude that their external product is again a spectral metric space?}
\medskip

At this level of generality we have been unable to provide a satisfactory answer and this stems from a slight lack of information regarding matricial stabilization of spectral metric spaces. Recall in this respect that a spectral triple $(\C A,H,D)$ can be stabilized by $(n \ti n)$-matrices yielding the new spectral triple $(M_n(\C A),H^{\op n}, D^{\op n})$. Hence, we do not only have a canonical seminorm on $\C A$ but also a canonical seminorm on all the matrix algebras over $\C A$ coming from norms of commutators with the direct sum of the abstract Dirac operator with itself. These matricial stabilizations already play a pivotal role in the development of unbounded $KK$-theory, \cite{Mes:UCN,KaLe:SFU,MeRe:NST}, and it is therefore not surprising that they are relevant in the study of external products of spectral metric spaces as well. In unbounded $KK$-theory, the central construction is the internal unbounded Kasparov product which is a far reaching generalization of the external product of unital spectral triples.  

When considering the metric aspects of an external product of two spectral metric spaces, our first task is thus to refine the notion of a compact quantum metric space by taking into account the state spaces of all the matricial stabilizations. This endeavor naturally takes place within the framework of operator systems since matrices over operator systems again have canonical state spaces. Our search for an appropriate matricial refinement proceeds through a novel characterization of compact quantum metric spaces by means of finite dimensional approximations. The defect in accuracy of such a finite dimensional approximation is controlled by the seminorm under investigation. We remark in passing that the approximation methods introduced in this paper appear to be an efficient way of verifying that a given candidate is a compact quantum metric space.

Let us briefly explain how our matricial refinement works for spectral metric spaces, but emphasize that the refinement also applies to more general compact quantum metric spaces resulting in a novel concept called a \emph{matrix compact quantum metric space}.

Given a unital spectral triple $(\C A,H,D)$ we do not merely focus on finite dimensional approximations at the level of the coordinate algebra $\C A$ but instead on finite dimensional approximations which apply simultaneously to all finite matrices with entries in the coordinate algebra. We then say that our unital spectral triple is a \emph{matrix spectral metric space}, when it can be approximated arbitrarily well by finite dimensional data in a uniform fashion where all the matrix algebras are taken into account. For each $n \in \nn$, the defect in accuracy at the matrix algebra level is measured by the seminorm coming from the commutator interaction between $M_n(\C A)$ and the $n$-fold direct sum $D^{\op n}$. Using this kind of matricial approximation we establish the following:

\begin{thm}
The external product of two matrix spectral metric spaces is again a matrix spectral metric space. 
\end{thm}

In order to substantiate our definition of a matrix compact quantum metric space we provide several noncommutative examples coming from ergodic actions, noncommutative tori and the Podle\'s sphere. These examples already appear in \cite{Rie:MSA,Li:DCM,AgKa:PSM,AKK:PSC}, but we are able to improve on these results by showing that they are examples of matrix compact quantum metric spaces instead of just compact quantum metric spaces.

In the papers \cite{AgBi:SHC,Agu:QTC}, the authors consider tensor products of continuous functions on compact metric spaces and AF-algebras and, among other things, they provide such tensor products with compact quantum metric space structures. We believe that the constructions and main theorems presented in this paper can be applied to put parts of \cite[Theorem 2.10]{AgBi:SHC} and \cite[Theorem 3.10]{Agu:QTC} on firmer conceptual grounds. At present we have decided to leave out a detailed investigation of these matters but hope to find time to revisit this relationship at a later stage.   

The present paper grew out of an urge to develop a bivariant theory of spectral metric data where the internal unbounded Kasparov product can be applied to construct new spectral metric spaces in a systematic and canonical fashion. Even though we are still far from seeing the contours of such an ambitious program, it seems that a natural first step in this direction is to investigate external products of spectral metric spaces and this is exactly what we are doing here. 

\subsection{Acknowledgements} 
The author gratefully acknowledge the financial support from the Independent Research Fund Denmark through grant no.~9040-00107B, 7014-00145B and 1026-00371B.

This work also benefited from many nice discussions with friends and colleagues, in particular Konrad Aguilar, David Kyed and Walter van Suijlekom. 

\subsection{Standing conventions} 
The notation $\| \cd \|$ will always refer to the unique $C^*$-norm on a $C^*$-algebra. To avoid treating irrelevant special cases all Hilbert spaces and $C^*$-algebras in this text are assumed to be non-trivial (different from $\{0\}$). 

For a Hilbert space $H$ we let $\B L(H)$ denote the unital $C^*$-algebra of bounded operators on $H$. 
%

For a compact Hausdorff space $M$ and a unital $C^*$-algebra $A$, the notation $C(M,A)$ refers to the unital $C^*$-algebra of continuous maps from $M$ to $A$. In the case where $A = \cc$ we often write $C(M)$ instead of $C(M,A)$.

We shall sometimes consider an extended metric $\rho : M \ti M \to [0,\infty]$ on a set $M$. This is a map which satisfies all the usual properties of a metric except that the distance between two points can be equal to infinity. Given a point $p$ in an (extended) metric space $M$ we let $\B B_r(p)$ and $\ov{\B B}_r(p)$ denote the open and closed balls, respectively, with center $p$ and radius $r \geq 0$. In particular, this notation applies to open and closed balls in normed vector spaces where the metric comes from the norm in the usual way.

We apply the convention that the supremum and the infimum of the empty set are equal to zero.

\section{Preliminaries}


\subsection{Operator systems}\label{ss:opsys}
We are here presenting a few basic definitions mainly relating to operator systems. The present text deals exclusively with concrete operator systems but we emphasize that the theory has a beautiful abstract counterpart, see \cite{ChEf:IOS}.

\begin{dfn}\label{d:opesys}
  An \emph{operator system} $\C X$ is a unital $*$-invariant subspace of a unital $C^*$-algebra $A$. We refer to the norm on $\C X$ inherited from $A$ as the \emph{$C^*$-norm} on $\C X$. 
  The operator system $\C X \su A$ is said to be \emph{complete} when it is closed in $C^*$-norm. 

  A linear subspace $\C Y \su \C X$ is called a sub-operator system of $\C X$ when $1_{\C X} \in \C Y$ and $y^* \in \C Y$ for all $y \in \C Y$.
\end{dfn}


Remark that a unital $C^*$-algebra $A$ becomes a complete operator system when considered as a unital $*$-invariant subspace of itself. 

For the rest of this subsection we fix an operator system $\C X \su A$. The $*$-invariance condition means that $x^* \in \C X$ for all $x \in \C X$ and we emphasize that the unit $1_{\C X}$ in the operator system $\C X$ has to agree with the unit in the unital $C^*$-algebra $A$. The $C^*$-norm closure of $\C X$ is denoted by $X$ and this is a complete operator system. We say that an element $x \in \C X$ is \emph{positive} when $x$ is positive as an element in the unital $C^*$-algebra $A$. 

For each $n \in \nn$, we can also regard the $(n \ti n)$-matrices with entries in $\C X$ as an operator system. Indeed, $M_n(\C X)$ can be thought of as a unital subspace of the unital $C^*$-algebra $M_n(A)$ and this subspace is again invariant under the adjoint operation. We may thus talk about completely bounded operators and completely positive operators acting between operator systems. 

\begin{dfn}\label{d:cpcb}
Let $\C X \su A$ and $\C Y \su B$ be operator systems and let $\al : \C X \to \C Y$ be a linear map. For each $n \in \nn$ we let $\al_n : M_n(\C X) \to M_n(\C Y)$ denote the linear map obtained by applying $\al$ entrywise. We say that $\al$ is
\begin{enumerate}
\item \emph{completely positive} when $\al_n : M_n(\C X) \to M_n(\C Y)$ is positive for all $n \in \nn$;
\item \emph{completely bounded} when there exists a constant $C \geq 0$ such that
\[
\| \al_n(x) \| \leq C \cd \| x \|
\]
for all $n \in \nn$ and all $x \in M_n(\C X)$;
\item \emph{completely contractive} when $\| \al_n(x) \| \leq \| x \|$ for all $n \in \nn$ and all $x \in M_n(\C X)$;
\item \emph{completely isometric} when $\| \al_n(x) \| = \| x \|$ for all $n \in \nn$ and all $x \in M_n(\C X)$.
\end{enumerate}
\end{dfn}

Notice that if $\al : \C X \to \C Y$ is a unital linear map between operator systems, then it holds that $\al$ is completely contractive if and only if $\al$ is completely positive, see e.g. the discussion in \cite[\S 1.3.2 and \S 1.3.3]{BlMe:OMO}. 

Remark also that if $\C Y$ sits as an operator system inside $C(M)$ where $M$ is a compact Hausdorff space, then every positive linear map $\al : \C X \to \C Y$ is automatically completely positive. 


\begin{dfn}\label{d:opesta}
For each $n \in \nn$ we let $\T{UCP}_n(\C X)$ denote the set of unital completely positive maps from $\C X$ to $M_n(\cc)$. We refer to $\T{UCP}_n(\C X)$ as the \emph{matrix state space} in degree $n$. The notation $\T{UCP}_\infty(\C X)$ refers to the disjoint union of matrix state spaces $\coprod_{n = 1}^\infty \T{UCP}_n(\C X)$.
\end{dfn}

Let $n \in \nn$. We equip the matrix state space $\T{UCP}_n(\C X)$ with the point-norm topology and record that $\T{UCP}_n(\C X)$ in this fashion becomes a compact Hausdorff space. For $n = 1$ we note that $\T{UCP}_1(\C X)$ agrees with the usual state space $S(\C X)$ equipped with the weak-$*$-topology. We define the unital positive linear map
\[
\io : \C X \to C\big( \T{UCP}_n(\C X), M_n(\cc) \big) \q \io(x)(\varphi) = \varphi(x) 
\]
and record that $\io$ is an isometry as soon as $n \geq 2$. For $n = 1$ it holds that $\| \io(x) \| = \| x \|$ whenever $x \in \C X$ is selfadjoint. It is also important to note that the matrix state spaces $\T{UCP}_n(\C X)$ and $\T{UCP}_n(X)$ are homeomorphic via the restriction map $\T{UCP}_n(X) \to \T{UCP}_n(\C X)$ (recalling here that $X \su A$ denotes the $C^*$-norm closure of $\C X \su A$). Remark that our matrix state space in degree $n$ is different from the noncommutative state space $\C S_n(X)$ considered in \cite[Section 2.3]{CoSu:TRO} (see the discussion in \cite[Section 2]{Ker:MQG}).
%

The operator system $\C X$ gives rise to an order unit space
\[
\C X_{\T{sa}} := \big\{ x \in \C X \mid x = x^* \big\},
\]
where the order unit is the unit $1_{\C X}$ from $\C X$ and the partial order is inherited from the partial order on the set of selfadjoint elements in $A$. The state space $S(\C X_{\T{sa}})$ for the order unit space is isomorphic to the state space $S(\C X)$ for the operator system. Indeed, we may extend a state $\varphi : \C X_{\T{sa}} \to \rr$ to a state $\wit{\varphi} : \C X \to \cc$ by putting $\wit{\varphi}(x) := \varphi\big(\T{re}(x)\big) + i \cd \varphi\big(\T{im}(x)\big)$, where $\T{re}(x)$ and $\T{im}(x)$ denote the real part and the imaginary part of $x$, respectively.

The above order unit space $\C X_{\T{sa}}$ is an example of a concrete order unit space. Similarly to the situation for operator systems there is also an abstract counterpart to the theory of order unit spaces, see \cite{Kad:RTC}. 
%
%
%
%

\subsection{Compact quantum metric spaces}
Throughout this subsection $\C X \su A$ is an operator system and $L : \C X \to [0,\infty)$ is a seminorm on $\C X$. The scalars $\cc$ are tacitly identified with the closed subspace $\B C \cd 1_{\C X} \su \C X$.  

The quotient norm on $\C X / \cc$ is denoted by $\| \cd \|_{\C X/\cc} : \C X / \cc \to [0,\infty)$ and the quotient map is denoted by $[\cd ] : \C X \to \C X / \cc$.

We are now going to introduce the notion of a compact quantum metric space. It is worthwhile to mention that the foundational work of Rieffel on compact quantum metric spaces was mainly carried out in the context of (abstract) order unit spaces, see for example \cite{Rie:MSS,Rie:GHQ}. In more recent times, the focus seems to change towards operator systems and this is indeed the framework we are using in the present text, see \cite{CoSu:STN,Sui:GSS,CoSu:TRO}. We believe that the operator system framework for compact quantum metric spaces was first treated in \cite{Ker:MQG}. Let us however emphasize that the operator systems appearing in this text are not assumed to be complete.

We define the seminorm closed ball with center $0$ and radius $r \geq 0$ by
\[
\ov{\B B}_r(0,L) := \big\{ x \in \C X \mid L(x) \leq r \big\} .
\]
The seminorm $L : \C X \to [0,\infty]$ is called \emph{lower semicontinuous} when it holds for each $r \geq 0$ that $\ov{\B B}_r(0,L)$ is closed as a subset of $\C X$ equipped with the topology coming from the $C^*$-norm.

\begin{dfn}\label{d:lipschitz}
We say that the seminorm $L : \C X \to [0,\infty)$ is a \emph{Lipschitz seminorm} when the following conditions hold:
\begin{enumerate}
\item $L$ is \emph{$*$-invariant}, meaning that $L(x^*) = L(x)$ for all $x \in \C X$;
\item $L$ \emph{vanishes on scalars}, meaning that $L(1_{\C X}) = 0$.
\end{enumerate}
\end{dfn}

The \emph{kernel} of the seminorm $L : \C X \to [0,\infty)$ is defined as the subspace
\[
\T{ker}(L) := \{ x \in \C X \mid L(x) = 0 \} \su \C X .
\]

\begin{remark}\label{r:kernel}
It is common to require that the kernel of a Lipschitz seminorm agrees with the scalars $\cc \su \C X$. In this text we only assume that $\cc \su \T{ker}(L)$ since the reverse inclusion need not hold for seminorms arising from unital spectral triples as described in Subsection \ref{ss:spemetspa}. Moreover, the reverse inclusion holds automatically if $(\C X,L)$ is a compact quantum metric space in the sense of Definition \ref{d:cqms}.
%
\end{remark}

Let us fix a Lipschitz seminorm $L : \C X \to [0,\infty)$. For each $n \in \nn$ we define the \emph{Monge-Kantorovich metric} on the matrix state space by
\[
\begin{split}
& \rho_L : \T{UCP}_n(\C X) \ti \T{UCP}_n(\C X) \to [0,\infty] \\ 
& \rho_L(\varphi,\psi) 
:= \sup\big\{ \| \varphi(x) - \psi(x)\| \mid x \in \ov{\B B}_1(0,L) \big\} .
\end{split}
\]
Notice that $\rho_L$ is allowed to take the value infinity but that $\rho_L$ satisfies all the remaining properties of a metric, thus $\rho_L$ is an extended metric. We are interested in the metric topology on the matrix state space $\T{UCP}_n(\C X)$ with basis consisting of all the open balls (with finite radii).

\begin{dfn}\label{d:cqms}
We say that the pair $(\C X,L)$ is a \emph{compact quantum metric space} when the Monge-Kantorovich metric $\rho_L$ metrizes the weak-$*$-topology on the state space $S(\C X) = \T{UCP}_1(\C X)$.
\end{dfn}

As mentioned in Remark \ref{r:kernel} it holds that if $(\C X,L)$ is a compact quantum metric space, then $\T{ker}(L) = \cc$. This implication can be proved using the argument from \cite[Lemma 2.2]{KaKy:DCQ}. 


The next result follows from \cite[Theorem 1.8]{Rie:MSA}, but see also \cite[Theorem 6.3]{Pav:DOA}. 

\begin{thm}\label{t:totbou}
Let $\C X$ be an operator system equipped with a Lipschitz seminorm $L : \C X \to [0,\infty)$. The following conditions are equivalent:
\begin{enumerate}
\item $(\C X,L)$ is a compact quantum metric space;
\item the image of the seminorm closed unit ball $\, \ov{\B B}_1(0,L)$ under the quotient map $[\cd ] : \C X \to \C X/ \cc$ is totally bounded with respect to the quotient norm.
\end{enumerate}
\end{thm}

The final result of this subsection is due to Kerr and provides information on the Monge-Kantorovich metric on matrix state spaces, \cite[Proposition 2.12]{Ker:MQG}.

\begin{prop}\label{p:metmatsta}
Let $\C X$ be a operator system equipped with a Lipschitz seminorm $L : \C X \to [0,\infty)$. The following conditions are equivalent:
\begin{enumerate}
\item $(\C X,L)$ is a compact quantum metric space;
\item the Monge-Kantorovich metric $\rho_L$ metrizes the point-norm topology on the matrix state space $\T{UCP}_n(\C X)$ for all $n \in \nn$. 
\end{enumerate}
\end{prop}

\subsection{Spectral metric spaces}\label{ss:spemetspa}
We are in this text particularly interested in compact quantum metric spaces arising from noncommutative geometric data. We therefore recall the fundamental notion of a unital spectral triple and explore the relationship between unital spectral triples and compact quantum metric spaces, see \cite{Con:CFH,Con:NCG,Con:GCM}.

%

\begin{dfn}\label{d:spetri}
Let $\C A$ be a unital $*$-subalgebra of the bounded operators on a separable Hilbert space $H$ and let $D : \T{dom}(D) \to H$ be a selfadjoint unbounded operator. We say that the triple $(\C A,H,D)$ is a \emph{unital spectral triple} when the following holds:
\begin{enumerate}
\item for every $a \in \C A$ and $\xi \in \T{dom}(D)$ it holds that $a(\xi) \in \T{dom}(D)$ and the commutator
  \[
[D,a] : \T{dom}(D) \to H
\]
extends to a bounded operator $d(a) : H \to H$;
\item the resolvent $(i + D)^{-1} : H \to H$ is a compact operator.
\end{enumerate}
A unital spectral triple is \emph{even} when the separable Hilbert space $H$ is equipped with a $\zz/2\zz$-grading operator $\ga : H \to H$ such that $a : H \to H$ is even for all $a \in \C A$ and $D : \T{dom}(D) \to H$ is odd. Otherwise, a unital spectral triple is \emph{odd}.

We emphasize that the unit in $\C A$ is required to agree with the unit in $\B L(H)$.
\end{dfn}

Many of the results in the present paper are correct for triples $(\C A,H,D)$ which only satisfy condition $(1)$ of Definition \ref{d:spetri}. We shall refer to such a triple as a \emph{unital Lipschitz triple}. For a unital Lipschitz triple $(\C A,H,D)$, the unital $*$-subalgebra $\C A \su \B L(H)$ may be considered as an operator system and the commutator interaction between $\C A$ and $D$ yields the seminorm
\[
L_D : \C A \to [0,\infty) \q L_D(a) := \| d(a) \|.
  \]
The next result regarding $L_D$ is now easily established, except perhaps for lower semicontinuity which is proved in \cite[Proposition 3.7]{Rie:MSS}:


\begin{lemma}\label{l:liplip}
Suppose that $(\C A,H,D)$ is a unital Lipschitz triple. Then the seminorm $L_D : \C A \to [0,\infty)$ is a lower semicontinuous Lipschitz seminorm. 
\end{lemma}

%
%

The following terminology comes from \cite{BMR:DSS}: 

\begin{dfn}
A \emph{spectral metric space} is a unital spectral triple $(\C A,H,D)$ such that the pair $(\C A,L_D)$ is a compact quantum metric space.
\end{dfn}

\begin{remark}
There are many examples of unital spectral triples $(\C A,H,D)$, where $(\C A,L_D)$ is not a compact quantum metric space. In fact, if $(\C A,H,D)$ is a unital spectral triple we may consider the unital spectral triple $\big(M_2(\C A), H^{\op 2}, D^{\op 2}\big)$. This stabilization of $(\C A,H,D)$ by $(2 \ti 2)$-matrices is never a spectral metric space because $M_2(\cc) \su \T{ker}(L_{D^{\op 2}})$.
\end{remark}

\section{Finite dimensional approximation}\label{s:findimapp}
Throughout this section we let $\C X \su A$ be an operator system equipped with a Lipschitz seminorm $L : \C X \to [0,\infty)$. We are going to characterize what it means for $(\C X,L)$ to be a compact quantum metric space in terms of a particular kind of finite dimensional approximations. We start out by briefly discussing the diameter of the state space $S(\C X)$ equipped with the Monge-Kantorovich metric $\rho_L$. Recall that $[\cd ] : \C X \to \C X/\cc$ denotes the quotient map. 

\begin{dfn}\label{d:bounded}
We say that the pair $(\C X,L)$ has \emph{finite diameter} when there exists a constant $C \geq 0$ such that
\[
\| [x] \|_{\C X / \cc} \leq C \cd L(x) \q \mbox{for all } x \in \C X .
\]
\end{dfn}

Remark that if $(\C X,L)$ has finite diameter, then the kernel of the Lipschitz seminorm $L : \C X \to [0,\infty)$ agrees with the subspace $\cc \su \C X$. 

The next result, which justifies the terminology introduced in Definition \ref{d:bounded}, can be found in \cite[Proposition 1.6]{Rie:MSA} and in \cite[Proposition 2.2]{Rie:MSS}. Notice that Rieffel also relates the diameter of the state space to the infimum of the constants satisfying the inequality in Definition \ref{d:bounded}. 

\begin{prop}\label{p:bounded}
The following conditions are equivalent:
\begin{enumerate}
\item $(S(\C X), \rho_L)$ has finite diameter, meaning that $\sup\{ \rho_L(\mu,\nu) \mid \mu,\nu \in S(\C X) \} < \infty$;
\item $(\C X,L)$ has finite diameter.
 \end{enumerate}
\end{prop}

We now introduce the kind of finite dimensional approximations which can be applied to characterize compact quantum metric spaces.

\begin{dfn}\label{d:bouapp}
Let $\ep,C > 0$ be constants and let $\C Y \su B$ be an operator system. We say that a pair $(\io,\Phi)$ consisting of unital bounded operators $\io, \Phi : \C X \to \C Y$ is an \emph{$(\ep,C)$-approximation} of $(\C X,L)$ when the following holds:
\begin{enumerate}
\item $\frac{1}{C} \cd \| x \| \leq \| \io(x) \|$ for all $x \in \C X$;
\item the image of $\Phi$ is a finite dimensional subspace of $\C Y$; 
\item $\| \io(x) - \Phi (x) \| \leq \ep \cd L(x)$ for all $x \in \C X$.
\end{enumerate}

We say that an $(\ep,C)$-approximation $(\io,\Phi)$ of $(\C X,L)$ is \emph{positive} when $\Phi$ is a unital positive operator from $\C X$ to $\C Y$. In the case where $\io : \C X \to \C Y$ is a unital isometry we refer to $(\io,\Phi)$ as an \emph{isometric} $\ep$-approximation.
\end{dfn}

Notice that condition $(1)$ from Definition \ref{d:bouapp} implies that the unital bounded map $\io$ induces a linear bounded isomorphism $\io : \C X \to \io(\C X)$ and that the inverse is also bounded with operator norm dominated by the constant $C > 0$.

Let us discuss the situation where we start out with a compact metric space $(M,\rho)$ and a unital $C^*$-algebra $B$. Consider the unital $C^*$-algebra $C(M,B)$ of continuous maps from $M$ to $B$ and let $\T{Lip}(M,B) \su C(M,B)$ denote the unital $*$-subalgebra of Lipschitz maps. We regard $\T{Lip}(M,B) \su C(M,B)$ as an operator system and define the lower semicontinuous Lipschitz seminorm $L_\rho : \T{Lip}(M,B) \to [0,\infty)$ by
\[
L_\rho(f) := \sup\big\{ \frac{\|f(p) - f(q)\|}{\rho(p,q)} \mid p \neq q \big\}.
\]
Remark that $L_\rho(f)$ is simply the Lipschitz constant associated to a Lipschitz map $f : M \to B$. For $B = \cc$ it holds that $\big( \T{Lip}(M,\cc),L_\rho \big)$ is a compact quantum metric space and the corresponding Monge-Kantorovich metric on the state space $S\big(\T{Lip}(M,\cc) \big)$ recovers the metric on $M$, see e.g. the discussion in the introduction to \cite{Rie:MSS}. We emphasize that the pair $\big( \T{Lip}(M,B), L_\rho \big)$ is \emph{not} a compact quantum metric space for $B \neq \cc$ since the kernel of $L_\rho$ can be identified with $B$. 

The identity operator on $\T{Lip}(M,B)$ is denoted by $\T{id} : \T{Lip}(M,B) \to \T{Lip}(M,B)$.


\begin{lemma}\label{l:metapp}
Let $(M,\rho)$ be a compact metric space and let $B$ be a finite dimensional unital $C^*$-algebra. For every $\ep > 0$ there exists a positive isometric $\ep$-approximation $(\T{id}, \Phi_\ep)$ of $\big(\T{Lip}(M,B),L_\rho \big)$.
\end{lemma}
\begin{proof}
Let $\ep > 0$ be given. Use the compactness of the metric space $M$ to choose finitely many points $p_1,p_2,\ldots,p_n \in M$ such that
\[
M = \cup_{j = 1}^n \B B_\ep(p_j) .
\]
Afterwards, choose a partition of unity $\{ \phi_j \}_{j = 1}^n$ such that $\phi_j(p) = 0$ for all $p \notin \B B_{\ep}(p_j)$. We may arrange that all the functions in our partition of unity are Lipschitz functions. Define the unital positive operator
\[
\Phi_\ep : \T{Lip}(M,B) \to \T{Lip}(M,B) \q \Phi_\ep(f) := \sum_{j = 1}^n f(p_j) \cd \phi_j .
\]
The identity map is clearly a unital isometry and the image of $\Phi_\ep$ is finite dimensional so we focus on proving condition $(3)$ from Definition \ref{d:bouapp}. Let $p \in M$ and $f \in \T{Lip}(M,B)$ be given. The relevant estimate follows from the inequalities:
\[
\begin{split}
\big\| f(p) - \Phi_\ep(f)(p) \big\| & \leq \sum_{j = 1}^n \big\| f(p) - f(p_j) \big\| \cd \phi_j(p) 
\leq \sum_{j = 1}^n \rho(p,p_j) \cd L_\rho(f) \cd \phi_j(p) \\
& \leq \ep \cd L_\rho(f) \cd \sum_{j = 1}^n \phi_j(p) = \ep \cd L_\rho(f) . \qedhere
\end{split}
\]
\end{proof}

The main result of this section, which provides a characterization of compact quantum metric spaces in terms of finite dimensional approximations, can now be stated and proved.

\begin{thm}\label{t:charcqms}
Let $\C X \su A$ be an operator system and let $L : \C X \to [0,\infty)$ be a Lipschitz seminorm. The following conditions are equivalent:
\begin{enumerate}
\item $(\C X,L)$ is a compact quantum metric space;
\item $(\C X,L)$ has finite diameter and for every $\ep > 0$ there exists a positive isometric $\ep$-approximation of $(\C X,L)$.
\item $(\C X,L)$ has finite diameter and there exists a constant $C > 0$ such that for every $\ep > 0$ there exists an $(\ep,C)$-approximation of $(\C X,L)$.
\end{enumerate}
\end{thm}
\begin{proof}
The implication $(2) \rar (3)$ is clear so we focus on proving that $(1) \rar (2)$ and $(3) \rar (1)$.

Suppose that $(1)$ holds. Since $\big(S(\C X),\rho_L\big)$ is a compact metric space it does in particular have finite diameter and it therefore follows from Proposition \ref{p:bounded} that $(\C X,L)$ has finite diameter as well. We are going to construct a positive isometric $\ep$-approximation of $(\C X,L)$ for every $\ep > 0$. So let $\ep > 0$ be given.

We start out by noting that Proposition \ref{p:metmatsta} shows that $\big( \T{UCP}_2(\C X), \rho_L\big)$ is a compact metric space. Applying Lemma \ref{l:metapp} we may choose a positive isometric $\ep$-approximation $(\T{id},\Phi_\ep)$ of the pair $\Big(\T{Lip}\big(\T{UCP}_2(\C X), M_2(\cc)\big), L_{\rho_L} \Big)$. Recall here that the Lipschitz maps from $\T{UCP}_2(\C X)$ to $M_2(\cc)$ are considered as an operator system inside the unital $C^*$-algebra of continuous maps from $\T{UCP}_2(\C X)$ to $M_2(\cc)$.

Let us now consider the unital linear map 
\[
\io : \C X \to C\big(\T{UCP}_2(\C X), M_2(\cc)\big) \q \io(x)(\varphi) = \varphi(x)
\]
which we know is positive and isometric. For every $x \in \C X$ we record that $\io(x)$ is in fact a Lipschitz map and that $L_{\rho_L}(\io(x)) \leq L(x)$. This follows since
\[
\| \io(x)(\varphi) - \io(x)(\psi) \| = \| \varphi(x) - \psi(x) \| \leq \rho_L(\varphi,\psi) \cd L(x)
\]
for all unital completely positive maps $\varphi,\psi : \C X \to M_2(\cc)$, see also \cite[Lemma 3.2 and Theorem 4.1]{Rie:MSS}. In particular, we get that $\io$ factorizes through $\T{Lip}\big(\T{UCP}_2(\C X), M_2(\cc)\big)$ yielding a unital positive isometry $\io : \C X \to \T{Lip}\big( \T{UCP}_2(\C X), M_2(\cc)\big)$.

We claim that $(\io,\Phi_\ep \io)$ is a positive isometric $\ep$-approximation of $(\C X,L)$. We only verify condition $(3)$ from Definition \ref{d:bouapp} since the remaining claims follow immediately from the constructions. Let thus $x \in \C X$ be given. Using that $(\T{id},\Phi_\ep)$ is a positive isometric $\ep$-approximation of $\big(\T{Lip}\big(\T{UCP}_2(\C X), M_2(\cc)\big), L_{\rho_L} \big)$ together with the inequality $L_{\rho_L}(\io(x)) \leq L(x)$, we obtain that
\[
\big\| \io(x) - (\Phi_\ep \io)(x) \big\| \leq \ep \cd L_{\rho_L}\big(\io(x)\big) \leq \ep \cd L(x) .
\]
This ends the proof of the implication $(1) \rar (2)$.
\medskip

Suppose that $(3)$ holds and choose the corresponding constant $C > 0$ so that there exists an $(\ep,C)$-approximation of $(\C X,L)$ for all $\ep > 0$. By enlarging $C > 0$ if necessary, we may at the same time arrange that 
\begin{equation}\label{eq:profin}
\| [x] \|_{\C X/\cc} \leq C \cd L(x) \q \T{for all } x \in \C X .
\end{equation}
By Theorem \ref{t:totbou} it suffices to show that the image $\big[ \ov{\B B}_1(0,L)\big] \su \C X / \cc$ is totally bounded. 

Let $\ep > 0$ be given. We need to find finitely many points $x_1,x_2,\ldots,x_n$ inside the seminorm closed unit ball $\ov{\B B}_1(0,L)$ such that 
\begin{equation}\label{eq:proinc}
\big[ \ov{\B B}_1(0,L)\big] \su \bigcup_{j = 1}^n \B B_\ep\big(  [x_j] \big) .
\end{equation}
Choose an $\big(\frac{\ep}{4 C},C\big)$-approximation $(\io,\Phi)$ of $(\C X,L)$. Since both $\io : \C X \to \C Y$ and $\Phi : \C X \to \C Y$ are unital and bounded they descend to bounded operators $[\io]$ and $[\Phi] : \C X / \cc \to \C Y / \cc$. The corresponding operator norms are denoted by $\| [\Phi] \|_\infty$ and $\| [\io] \|_\infty$, respectively. We let $F \su \C Y / \cc$ denote the image of $[\Phi] : \C X / \cc \to \C Y / \cc$ so that $F$ is finite dimensional by assumption. For every $x \in \ov{\B B}_1(0,L)$ we obtain from \eqref{eq:profin} that
\[
\big\| [\Phi] [x] \big\|_{\C Y/ \cc} \leq \big\| [\Phi] \big\|_\infty \cd \big\| [x] \big\|_{\C X / \cc} \leq C \cd \big\| [\Phi] \big\|_\infty .
\]
This shows that the image $[\Phi] \big[ \ov{\B B}_1(0,L)\big]$ is a bounded subset of $F$. Since $F$ is finite dimensional we get that $[\Phi] \big[ \ov{\B B}_1(0,L)\big] \su F$ is in fact totally bounded and we may choose $x_1,x_2,\ldots,x_n \in \ov{\B B}_1(0,L)$ such that
\[
[\Phi] \big[ \ov{\B B}_1(0,L)\big] \su \bigcup_{j = 1}^n \B B_{\frac{\ep}{2C}}\big( [\Phi(x_j)] \big) .
\]
Let $x \in \ov{\B B}_1(0,L)$ be given. Choose a $j \in \{1,2,\ldots,n\}$ such that $\big\| [\Phi(x - x_j)] \big\|_{\C Y/ \cc} < \frac{\ep}{2C}$. To ease the notation, put $z_j := x - x_j$. Let us moreover choose a $\la \in \cc$ such that $\| \Phi(z_j) - \la \| < \frac{\ep}{2C}$. Since $(\io,\Phi)$ is an $\big(\frac{\ep}{4C},C\big)$-approximation of $(\C X,L)$ we then have that
\[
\begin{split}
\big\| [x - x_j] \big\|_{\C X/\cc} & \leq \| z_j - \la \| \leq C \cd \big\| \io(z_j - \la) \big\| \\
& \leq C \cd \big\| \io(z_j) - \Phi(z_j) \big\| + C \cd \big\| \Phi(z_j) - \la  \big\| \\
& < \frac{\ep}{4} \cd L(z_j) + \frac{\ep}{2} \leq \frac{\ep}{4} \cd L(x) + \frac{\ep}{4} \cd L(x_j) + \frac{\ep}{2} \leq \ep .
\end{split}
\]
This proves the inclusion in \eqref{eq:proinc} and hence that $(3) \rar (1)$.
\end{proof}

We end this section by presenting two small results regarding the existence of finite dimensional approximations and the finite diameter condition.

\begin{lemma}\label{l:findia}
Let $L : \C X \to [0,\infty)$ be a Lipschitz seminorm. Suppose that $(\io,\Phi)$ is an $(\ep,C)$-approximation of $(\C X,L)$ for some constants $\ep,C > 0$ and suppose that there exists a norm $\vertiii{\cd} : \Phi(\C X)/\cc \to [0,\infty)$ and a constant $D \geq 0$ such that
\[
\vertiii{ [\Phi(x)]} \leq D \cd L(x) \q \mbox{for all } x \in \C X .
\]
Then $(\C X,L)$ has finite diameter. 
\end{lemma}
\begin{proof}
Since the image of $\Phi : \C X \to \C Y$ is finite dimensional, we know that the quotient space $\Phi(\C X)/\cc$ is finite dimensional as well. The norm $\vertiii{ \cd }$ is therefore equivalent to the quotient norm $\| \cd \|_{\Phi(\C X)/\cc}$ and we may choose a constant $E > 0$ such that $\| [y] \|_{\Phi(\C X)/\cc} \leq E \cd \vertiii{ [y] }$ for all $y \in \Phi(\C X)$. Let now $x \in \C X$ and $\de > 0$ be given. Choose a $\la \in \cc$ such that $\|\Phi(x) - \la \| \leq \| [\Phi(x)] \|_{\Phi(\C X)/\cc} + \de$ and notice that we have the estimates
\[
\begin{split}
\| [x] \|_{\C X/\cc} & \leq \| x - \la \| \leq C \cd \| \io(x)  - \la \| \leq C \cd \| \io(x) - \Phi(x) \| + C \cd \| \Phi(x) - \la \| \\
& \leq C \cd \ep \cd L(x) + C \cd \| [\Phi(x)] \|_{\Phi(\C X)/\cc} + C \cd \de \\
& \leq  C \cd \ep \cd L(x) + C \cd E \cd \vertiii{ [\Phi(x)] } + C \cd \de 
\leq C \cd ( \ep + D \cd E) \cd L(x)  + C \cd \de.
\end{split}
\]
Since $\de > 0$ was arbitrary we get that $\| [x] \|_{\C X/\cc} \leq C \cd ( \ep + D \cd E) \cd L(x)$. This proves the lemma.
\end{proof}

\begin{prop}
Let $L : \C X \to [0,\infty)$ be a Lipschitz seminorm with $\T{ker}(L) = \cc$. Suppose that $(\io,\Phi)$ is an $(\ep,C)$-approximation of $(\C X,L)$ for some constants $\ep,C > 0$. Suppose moreover that the codomain of both $\io$ and $\Phi$ is equal to $\C X$ and that $\Phi : \C X \to \C X$ is bounded with respect to $L$. Then $(\C X,L)$ has finite diameter.
\end{prop}
\begin{proof}
  Since $\Phi : \C X \to \C X$ is bounded with respect to $L$ we know that there exists a constant $D \geq 0$ such that $L(\Phi(x)) \leq D \cd L(x)$ for all $x \in \C X$. Choose a state $\varphi : \C X \to \cc$. Using that $\T{ker}(L) = \cc$ we define a norm
  \[
\vertiii{\cd} : \Phi(\C X)/\cc \to [0,\infty)
  \]
  by putting $\vertiii{[z]} := L\big( z - \varphi(z) \cd 1_{\C X} \big)$ for all $z \in \Phi(\C X)$. The result of the proposition now follows from Lemma \ref{l:findia} by noting that
  \[
  \vertiii{ [\Phi(x)]} = L\big( \Phi(x) - \varphi(\Phi(x)) \cd 1_{\C X} \big) \leq L\big( \Phi(x) \big) \leq D \cd L(x)
  \]
  for all $x \in \C X$. 
\end{proof}

\section{Operator seminorms and minimal tensor products}
We are in this text interested in the compatibility between the theory of compact quantum metric spaces and minimal tensor products of operator systems. We first study seminorms which take into account all tensor products with finite scalar matrices. We shall then see that these seminorms are compatible with minimal tensor products in full generality. The idea is to consider operator systems which are equipped with seminorms satisfying properties analoguous to the $L^\infty$-matrix norms which characterize an abstract operator space, see \cite{Rua:SCA,EfRu:ACO}. This kind of seminorms are called operator seminorms following existing terminology for operator spaces, see \cite[\S 1.2.16]{BlMe:OMO}. 

\subsection{Operator seminorms}\label{ss:matsem}
Let us fix an operator system $\C X \su A$.

For each $n \in \nn$ we recall from Subsection \ref{ss:opsys} that the $(n \ti n)$-matrices $M_n(\C X) \su M_n(A)$ is an operator system in its own right. Notice also that $M_n(\C X)$ is a bimodule over the scalar matrices $M_n(\cc)$ where the bimodule structure is given by
\begin{equation}\label{eq:module}
(v \cd x \cd w)_{ij} := \sum_{k,l = 1}^n v_{ik} \cd x_{kl} \cd w_{lj} \q i,j \in \{1,2,\ldots,n\}
\end{equation}
whenever $v,w \in M_n(\cc)$ and $x \in M_n(\C X)$. The subscript notation applied here refers to the entries of the corresponding $(n \ti n)$-matrix.

For each $n,m \in \nn$ and matrices $x \in M_n(\C X)$ and $y \in M_m(\C X)$ we form the block diagonal direct sum $x \op y \in M_{n + m}(\C X)$ with $x$ in the upper left corner and $y$ in the lower right corner.

\begin{dfn}\label{d:opesemi}
An \emph{operator seminorm} $\B L$ on $\C X$ consists of a seminorm $\B L_s : M_s(\C X) \to [0,\infty)$ for every $s \in \nn$ such that the following holds:
\begin{enumerate}
\item $\B L_{s+r}(x \op y) = \max\{ \B L_s(x), \B L_r(y) \}$ for all $s,r \in \nn$ and $x \in M_s(\C X)$ and $y \in M_r(\C X)$;
\item $\B L_s( v \cd x \cd w) \leq \| v \| \cd \B L_s(x) \cd \| w \|$ for all $s \in \nn$ and $x \in M_s(\C X)$ and $v,w \in M_s(\B C)$.
\end{enumerate}
We say that an operator seminorm $\B L$ is \emph{lower semicontinuous} when the seminorm $\B L_s : M_s(\C X) \to [0,\infty)$ is lower semicontinuous for all $s \in \nn$. 
\end{dfn}

Let us from now on fix an operator seminorm $\B L = \{\B L_s\}_{s = 1}^\infty$ on the operator system $\C X$. Consider a matrix $x \in M_s(\C X)$ for some $s \in \nn$. From the proof of \cite[Proposition 2.1]{Rua:SCA} we obtain the inequalities
\begin{equation}\label{eq:semientry}
\B L_s(x) \leq \sum_{i,j = 1}^s \B L_1(x_{ij}) \, \, \T{ and } \, \, \, \B L_1(x_{kl}) \leq \B L_s(x)
\end{equation}
for all $k,l \in \{1,2,\ldots,s\}$. In particular, we get the following result regarding kernels:

\begin{lemma}\label{l:matker}
For each $s \in \nn$ we have the identity $\T{ker}(\B L_s) = M_s\big( \T{ker}(\B L_1) \big)$.
\end{lemma}

The operator version of Lipschitz seminorms can now be introduced:

\begin{dfn}\label{d:lipmatsem}
We say that the operator seminorm $\B L = \{ \B L_s\}_{s = 1}^\infty$ is a \emph{Lipschitz operator seminorm} when the following conditions hold:
\begin{enumerate}
\item $\B L$ is \emph{$*$-invariant}, meaning that $\B L_s(x^*) = \B L_s(x)$ for all $s \in \nn$ and all $x \in M_s(\C X)$;
\item $\B L$ \emph{vanishes on scalar matrices}, meaning that $\B L_s(v) = 0$ for all $s \in \nn$ and all $v \in M_s(\cc \cd 1_{\C X})$. 
\end{enumerate}
\end{dfn}

As a consequence of Lemma \ref{l:matker} we get that the operator seminorm $\B L$ is a Lipschitz operator seminorm if and only if $\B L$ is $*$-invariant and $\B L_1(1_{\C X}) = 0$. 
\medskip

Let us fix an $n \in \nn$. We shall now see how to stabilize the operator seminorm $\B L$ by $(n \ti n)$-matrices and thereby obtain a new operator seminorm $\B L^{M_{n}(\C X)}$ on the operator system $M_{n}(\C X) \su M_{n}(A)$. 

For each $s \in \nn$ we start out by considering the $*$-isomorphism of unital $C^*$-algebras $I_s : M_s\big( M_{n}(A) \big) \to M_{s \cd n}(A)$ which forgets the subdivisions. More precisely, for each $i,j \in \{1,2,\ldots,s\}$ and $k,l \in \{1,2,\ldots,n\}$ we have that
\[
I_s( x )_{ (i-1) n + k, (j-1)n + l} := ( x_{ij} )_{kl} \q x \in M_s\big( M_{n}(A)\big) .
\]
The $*$-isomorphism $I_s$ then restricts to a unital completely isometric isomorphism
\begin{equation}\label{eq:matsemi}
I_s : M_s\big( M_{n}(\C X)\big) \to M_{s \cd n}(\C X) .
\end{equation}
This latter isomorphism is also well-behaved with respect to the $M_s(\cc)$-bimodule structure on $M_s\big( M_{n}(\C X) \big)$ described in \eqref{eq:module}. Indeed, letting $\io : \cc \to M_{n}(\cc)$ denote the unique unital linear map we get that
\[
I_s( v \cd x \cd w) = I_s\big( \io_s(v) \big) \cd I_s(x) \cd I_s\big( \io_s(w) \big)
\]
for all $x \in M_s\big( M_{n}(\C X) \big)$ and $v,w \in M_s(\cc)$. Notice here that the products on the right hand side use the $M_{s \cd n}(\cc)$-bimodule structure on $M_{s \cd n}(\C X)$. We may also consider the block diagonal direct sum of matrices and notice that
\[
I_{s + r}( x \op y) = I_s(x) \op I_r(y)
\]
whenever $s,r \in \nn$ and $x \in M_s( M_{n}(\C X))$ and $y \in M_r( M_{n}(\C X))$. 

For each $s \in \nn$ we define the stabilized seminorm $\B L_s^{M_{n}(\C X)} : M_s\big( M_{n}(\C X) \big) \to [0,\infty)$ as the composition of the identification in \eqref{eq:matsemi} and the seminorm $\B L_{s \cd n} : M_{s \cd n}(\C X) \to [0,\infty)$. 

Gathering the above observations we obtain the following result:

\begin{lemma}\label{l:matrixstab}
The sequence of seminorms $\B L^{M_{n}(\C X)} := \{ \B L_s^{M_{n}(\C X)} \}_{s = 1}^\infty$ is an operator seminorm on the operator system $M_{n}(\C X)$. Moreover, if the operator seminorm $\B L$ is Lipschitz, then $\B L^{M_{n}(\C X)}$ is Lipschitz and if $\B L$ is lower semicontinuous, then $\B L^{M_{n}(\C X)}$ is again lower semicontinuous.
\end{lemma}

It is worthwhile to point out that Lemma \ref{l:matker} implies that the kernel of the seminorm $\B L_1^{M_n(\C X)} : M_n(\C X) \to [0,\infty)$ agrees with the subspace $M_n(\T{ker}(\B L_1)) \su M_n(\C X)$. In particular, we see that if the operator seminorm $\B L = \{\B L_s\}_{s = 1}^\infty$ is Lipschitz, then the kernel of $\B L_1^{M_{n}(\C X)}$ does at least contain the subspace $M_{n}(\B C \cd 1_{\C X}) \su M_{n}(\C X)$. 


\subsection{Minimal tensor products}\label{ss:mtp}
Throughout this subsection we let $\C X \su A$ and $\C Y \su B$ be operator systems. We are moreover fixing operator seminorms $\B L := \{\B L_s\}_{s = 1}^\infty$ and $\B K := \{ \B K_s\}_{s = 1}^\infty$ on $\C X$ and $\C Y$, respectively. 

We denote the minimal tensor product of the unital $C^*$-algebras $A$ and $B$ by $A \otm B$ and consider the algebraic tensor product
\[
\C X \ot \C Y := \T{span}_{\cc}\big\{ x \ot y \mid x \in \C X \, , \, \, y \in \C Y \big\} \su A \otm B .
\]
This algebraic tensor product is again an operator system.

Let $n,m \in \nn$ and let $\varphi : \C Y \to M_n(\cc)$ and $\psi : \C X \to M_m(\cc)$ be unital completely positive maps. As a consequence of \cite[Theorem 4.4]{KPTT:TPO} we get that
\begin{equation}\label{eq:oneucp}
\begin{split}
& 1 \ot \varphi : \C X \ot \C Y \to M_n(\C X) \q (1 \ot \varphi)(x \ot y) := x^{\op n} \cd \varphi(y) \q \T{and} \\
  & \psi \ot 1 : \C X \ot \C Y \to M_m(\C Y) \q (\psi \ot 1)(x \ot y) := \psi(x) \cd y^{\op m}
  \end{split}
\end{equation}
  are unital completely positive maps.

For each $s \in \nn$ and $z \in M_s(\C X \ot \C Y)$ we may then describe the $C^*$-norm of $z$ viewed as an element in the unital $C^*$-algebra $M_s(A \otm B)$ via the following two suprema:
  \begin{equation}\label{eq:c*norm}
  \begin{split}
  \| z \| & = \sup\Big\{ \big\| (1 \ot \varphi)_s(z) \big\| \mid \varphi \in \T{UCP}_\infty(\C Y) \Big\} \\
  & = \sup\Big\{ \big\| (\psi \ot 1)_s(z) \big\| \mid \psi \in \T{UCP}_\infty(\C X) \Big\} .
  \end{split}
  \end{equation}
  Notice here that for $n \in \nn$ and $\varphi \in \T{UCP}_n(\C X)$ we get that $(1 \ot \varphi)_s(z) \in M_s( M_n(\C X))$ and the corresponding $C^*$-norm is then inherited from the unital $C^*$-algebra, $M_s( M_n(A) )$. This explains the various $C^*$-norms appearing in the first supremum of \eqref{eq:c*norm}. The $C^*$-norms appearing in the second supremum of \eqref{eq:c*norm} can be explained in a similar fashion. 

For later use we spell out a small observation regarding minimal tensor products: For each $n \in \nn$ we have the linear isomorphism
\begin{equation}\label{eq:matrix}
\io : M_n(\C X) \ot \C Y \to M_n(\C X \ot \C Y) \q \io(x \ot y)_{ij} := x_{ij} \ot y
\end{equation}
and, upon viewing the two algebraic tensor products as operator systems inside the corresponding minimal tensor products of unital $C^*$-algebras, we get that $\io$ is a unital complete isometry.
\medskip

The aim of this subsection is to introduce operator seminorms $\B L \ot 1$ and $1 \ot \B K$ on the operator system $\C X \ot \C Y \su A \otm B$. 
\medskip

Let us first fix two strictly positive integers $n, m \in \nn$ together with two unital completely positive maps $\varphi : \C Y \to M_n(\cc)$ and $\psi : \C X \to M_m(\cc)$. For each $s \in \nn$ we recall the definition of the stabilized seminorms
\[
\B L^{M_n(\C X)}_s : M_s\big( M_n(\C X) \big) \to [0,\infty) \, \, \T{and} \, \, \,
\B K^{M_m(\C Y)}_s : M_s\big( M_m(\C Y)\big) \to [0,\infty) .
  \]
  from Subsection \ref{ss:matsem}. We then define seminorms $\B L^\varphi_s$ and $\B K^\psi_s : M_s( \C X \ot \C Y ) \to [0,\infty)$ by putting
\[ 
\B L^{\varphi}_s(z) := \B L_s^{M_n(\C X)}\big( (1 \ot \varphi)_s(z) \big) \, \, \T{ and } \, \, \,
\B K^{\psi}_s(z) := \B K_s^{M_m(\C Y)}\big((\psi \ot 1)_s(z)\big)
\]
for all $z \in M_s( \C X \ot  \C Y )$.


The next lemma is a consequence of Lemma \ref{l:matrixstab}.

\begin{lemma}\label{l:ucp}
The sequences of seminorms $\B L^\varphi := \big\{ \B L^\varphi_s \big\}_{s = 1}^\infty$ and $\B K^\psi := \big\{ \B K^\psi_s \big\}_{s = 1}^\infty$ are operator seminorms on the operator system $\C X \ot \C Y$. If $\B L$ is Lipschitz, then $\B L^\varphi$ is Lipschitz and if $\B L$ is lower semicontinuous, then $\B L^\varphi$ is also lower semicontinuous. Similar statements hold for the operator seminorms $\B K$ and $\B K^\psi$.
\end{lemma}

In the following lemma we estimate the operator seminorms $\B L^\varphi$ and $\B K^\psi$ on elementary tensors.

\begin{lemma}\label{l:upphipsi}
For each $x \in \C X$ and $y \in \C Y$ we have the inequalities
  \[
\B L^\varphi_1( x \ot y ) \leq \B L_1(x) \cd \| y \| \q \mbox{and} \q \B K^\psi_1(x \ot y) \leq \| x \| \cd \B K_1(y) .
  \]
\end{lemma}
\begin{proof}
Using the defining properties of operator seminorms we get that
\[
\begin{split}
\B L^\varphi_1( x \ot y ) & = \B L_1^{M_n(\C X)}\big( x^{\op n} \cd \varphi(y) \big) = \B L_n\big(x^{\op n} \cd \varphi(y) \big) \\
& \leq \B L_n(x^{\op n}) \cd \| \varphi(y) \| = \B L_1(x) \cd \| \varphi(y) \| \leq \B L_1(x) \cd \| y \| .
\end{split}
\]
A similar argument applies to the seminorm $\B K^\psi_1 : \C X \ot \C Y \to [0,\infty)$.
\end{proof}

For each $s \in \nn$ we define the seminorms $(\B L \ot 1)_s$ and $(1 \ot \B K)_s$ on $M_s(\C X \ot \C Y)$ by the formulae
\[
\begin{split}
& (\B L \ot 1)_s(z) := \sup\big\{ \B L^\varphi_s(z) \mid \varphi \in \T{UCP}_\infty(\C Y) \big\} \q \T{and} \\
& (1 \ot \B K)_s(z) := \sup\big\{ \B K^\psi_s(z) \mid \psi \in \T{UCP}_\infty(\C X) \big\} .
\end{split}
\]
Notice that Lemma \ref{l:upphipsi} together with the first inequality in \eqref{eq:semientry} imply that neither $(\B L \ot 1)_s$ nor $(1 \ot \B K)_s$ can take the value infinity. 

The main result of this subsection can now be stated. The proof is an application of Lemma \ref{l:ucp} and the estimates in Lemma \ref{l:upphipsi}. Remark that the statement regarding lower semicontinuity follows since a supremum of lower semicontinuous seminorms is again a lower semicontinuous seminorm. 

\begin{lemma}\label{l:tenslip}
  The sequences of seminorms $\B L \ot 1 := \big\{ (\B L \ot 1)_s \big\}_{s = 1}^\infty$ and $1 \ot \B K := \big\{ (1 \ot \B K)_s \big\}_{s = 1}^\infty$ are operator seminorms on the operator system $\C X \ot \C Y$. Moreover, for every $x \in \C X$ and $y \in \C Y$ we have the inequalities
  \begin{equation}\label{eq:tenslip}
(\B L \ot 1)_1(x \ot y) \leq \B L_1(x) \cd \| y \| \q \mbox{and} \q (1 \ot \B K)_1(x \ot y) \leq \| x \| \cd \B K_1(y) .
  \end{equation}
  If $\B L$ is Lipschitz, then $\B L \ot 1$ is Lipschitz and if $\B L$ is lower semicontinuous, then $\B L \ot 1$ is again lower semicontinuous. Similar statements hold for the operator seminorms $\B K$ and $1 \ot \B K$.
\end{lemma}

We emphasize that if the operator seminorm $\B L$ is Lipschitz, then the kernel of the seminorm $(\B L \ot 1)_1 : \C X \ot \C Y \to [0,\infty)$ is in general much bigger than $\cc \cd (1_{\C X} \ot 1_{\C Y})$. Indeed, this kernel does at least contain the subspace $\cc \cd 1_{\C X} \ot \C Y$ of the minimal tensor product $\C X \ot \C Y$. Similarly, if the operator seminorm $\B K$ is Lipschitz, then the kernel of the seminorm $(1 \ot \B K)_1 : \C X \ot \C Y \to [0,\infty)$ contains the subspace $\C X \ot \cc \cd 1_{\C Y} \su \C X \ot \C Y$.
\medskip

For future applications it may be relevant to enlarge the domain of the operator seminorms $\B L \ot 1$ and $1 \ot \B K$ and we will now explain how this can be done. We suppose that both of the operator seminorms $\B L$ and $\B K$ are $*$-invariant in the sense of Definition \ref{d:lipmatsem} $(1)$. The two new operator seminorms are going to be denoted by $\B L \hot 1$ and $1 \hot \B K$ and they will be defined on operator systems $\C D(\B L \hot 1) \su A \otm B$ and $\C D(1 \hot \B K) \su A \otm B$, respectively. Both of these operator systems are going to contain the algebraic tensor product $\C X \ot \C Y$.

Let us denote the $C^*$-norm closure of the algebraic tensor product $\C X \ot \C Y \su A \otm B$ by $X \otm Y$ so that $X \otm Y$ is a complete operator system inside $A \otm B$. Furthermore, we let $X$ and $Y$ denote the $C^*$-norm closures of $\C X \su A$ and $\C Y \su B$, respectively. For each $n \in \nn$ and each $\varphi \in \T{UCP}_n(\C X)$ we record that the unital completely positive map $1 \ot \varphi : \C X \ot \C Y \to M_n(\C X)$ from \eqref{eq:oneucp} extends by continuity to a unital completely positive map $1 \ot \varphi : X \otm Y \to M_n(X)$.

      We introduce the sub-operator system $\C D(\B L \hot 1)$ of $X \otm Y$ by saying that an element $z \in X \otm Y$ belongs to $\C D(\B L \hot 1)$ if and only if the following two conditions hold:
      \begin{enumerate}
      \item For each $n \in \nn$ and $\varphi \in \T{UCP}_n(\C X)$ we have that $(1 \ot \varphi)(z) \in M_n(\C X)$;
      \item There exists a constant $C \geq 0$ (depending on $z$) such that
        \[
        \B L_n\big( (1 \ot \varphi)(z) \big) \leq C \q \T{for all } n \in \nn \T{ and } \varphi \in \T{UCP}_n(\C X).
        \]
      \end{enumerate}
      
      For each $s \in \nn$ we then define the seminorm $(\B L \hot 1)_s : M_s\big( \C D(\B L \hot 1) \big) \to [0,\infty)$ by putting
      \[
(\B L \hot 1)_s(z) := \sup\big\{ \B L_s^{M_n(\C X)}\big( (1 \ot \varphi)_s(z) \big) \mid n \in \nn \T{ and } \varphi \in \T{UCP}_n(\C X) \big\}.
      \]
      Remark that a combination of the definition of the operator system $\C D(\B L \hot 1) \su A \otm B$ and the first inequality in \eqref{eq:semientry} implies that the above supremum is indeed finite. 

      Using an almost identical procedure we also obtain a sub-operator system $\C D(1 \hot \B K) \su X \otm Y$ and a seminorm $(1 \hot \B K)_s : M_s\big( \C D(1 \hot \B K) \big) \to [0,\infty)$ for every $s \in \nn$.
        
The following lemma is the analogue of Lemma \ref{l:tenslip} for our new sequences of seminorms $\B L \hot 1 := \big\{ (\B L \hot 1)_s \big\}_{s = 1}^\infty$ and $1 \hot \B K := \big\{ (1 \hot \B K)_s \big\}_{s = 1}^\infty$. It is moreover made explicit that $\B L \hot 1$ and $1 \hot \B K$ extend the sequences of seminorms $\B L \ot 1$ and $1 \ot \B K$, respectively. 
      
      \begin{lemma}\label{l:tensmax}
        Suppose that the operator seminorms $\B L$ and $\B K$ are $*$-invariant. The sequences of seminorms $\B L \hot 1 := \big\{ (\B L \hot 1)_s \big\}_{s = 1}^\infty$ and $1 \hot \B K := \big\{ (1 \hot \B K)_s \big\}_{s = 1}^\infty$ are operator seminorms on the operator systems $\C D(\B L \hot 1)$ and $\C D(1 \hot \B K)$, respectively. Moreover, for every $s \in \nn$ we have the inclusion
        \[
M_s\big( \C X \ot \C Y \big) \su M_s\big( \C D(\B L \hot 1) \cap \C D(1 \hot \B K) \big)
\]
and the identities
\[
(\B L \ot 1)_s(z) = (\B L \hot 1)_s(z) \q \mbox{and} \q (1 \ot \B K)_s(z) = (1 \hot \B K)_s(z)
\]
hold for all $z \in M_s(\C X \ot \C Y)$. If $\B L$ is Lipschitz, then $\B L \hot 1$ is Lipschitz and if $\B L$ is lower semicontinuous, then $\B L \hot 1$ is again lower semicontinuous. Similar statements hold for the operator seminorms $\B K$ and $1 \hot \B K$.
\end{lemma}

\subsection{Lipschitz operator seminorms from unital spectral triples}\label{ss:lipopespe}
Part of the motivation for introducing operator seminorms stems from the fact that any unital spectral triple (or more generally any unital Lipschitz triple) gives rise to a lower semicontinuous Lipschitz operator seminorm in a canonical way. In this subsection we shall see how this is achieved and we are moreover going to investigate the tensor product construction from Subsection \ref{ss:mtp} in the context of unital Lipschitz triples.

Let us a fix a unital Lipschitz triple $(\C A,H,D)$. For each $n \in \nn$ we may stabilize $(\C A,H,D)$ by $(n \ti n)$-matrices and obtain the unital Lipschitz triple $( M_n(\C A), H^{\op n}, D^{\op n})$. In the case where $(\C A,H,D)$ is a unital spectral triple, the stabilization $( M_n(\C A), H^{\op n}, D^{\op n})$ is again a unital spectral triple of the same parity as $(\C A,H,D)$. If $(\C A,H,D)$ is even, then the $\zz/2\zz$-grading operator for the stabilization is the $n$-fold direct sum of the original $\zz/2\zz$-grading operator. 

As in Subsection \ref{ss:spemetspa} we obtain a lower semicontinuous Lipschitz seminorm $L_{D^{\op n}} : M_n(\C A) \to [0,\infty)$ for every $n \in \nn$. The following fundamental result is then a consequence of Lemma \ref{l:liplip} (except for a few minor observations):

\begin{prop}
The sequence of seminorms $\B L_D := \{ L_{D^{\op n}} \}_{n = 1}^\infty$ defines a lower semicontinuous Lipschitz operator seminorm on the operator system $\C A \su \B L(H)$.
\end{prop}

Let $G$ be another separable Hilbert space and consider the essentially selfadjoint unbounded operators
\[
D \ot 1 : \T{dom}(D) \ot G \to H \ot_2 G \q \T{and} \q 1 \ot D : G \ot \T{dom}(D) \to G \ot_2 H,
\]
where we emphasize that the domains are given by algebraic tensor products whereas the codomains are completed Hilbert space tensor products. We denote the closures by $D \hot 1 : \T{dom}(D \hot 1) \to H \ot_2 G$ and $1 \hot D : \T{dom}(1 \hot D) \to G \ot_2 H$, respectively.

\begin{prop}\label{p:stable}
Let $(\C A,H,D)$ be a unital Lipschitz triple and let $\C B \su \B L(G)$ be a unital $*$-subalgebra. The triples $(\C A \ot \C B, H \ot_2 G, D \hot 1)$ and $(\C B \ot \C A, G \ot_2 H, 1 \hot D)$ are unital Lipschitz triples and we have the identities
\[
\B L_{D \hot 1} = \B L_D \ot 1 \q \mbox{and} \q 1 \ot \B L_D = \B L_{1 \hot D}
\]
of operator seminorms on the operator systems $\C A \ot \C B \su \B L(H \ot_2 G)$ and $\C B \ot \C A \su \B L(G \ot_2 H)$, respectively.
\end{prop}
\begin{proof}
  We focus on proving the identity $L_{(D \hot 1)^{\op s}} = (\B L_D \ot 1)_s$ for every $s \in \nn$ since the remaining claims are either easily established or follow by similar arguments. 

  We recall that the derivation coming from the unital Lipschitz triple $(\C A,H,D)$ is denoted by $d : \C A \to \B L(H)$ so that $d(a)$ agrees with the closure of the commutator $[D,a] : \T{dom}(D) \to H$ for all $a \in \C A$. The derivation $d$ induces a derivation $d \ot 1 : \C A \ot \C B \to \B L(H) \ot \C B$ given by $(d \ot 1)(a \ot b) = d(a) \ot b$ for all $a \in \C A$ and $b \in \C B$.

  Consider now a unital completely positive map $\varphi : \C B \to M_n(\cc)$ for some $n \in \nn$ and notice that we have the identities
  \[
d_n (1 \ot \varphi)(a \ot b)  = d_n\big( a^{\op n} \cd \varphi(b) \big) = d(a)^{\op n} \cd \varphi(b) = (1 \ot \varphi)(d \ot 1)(a \ot b)
\]
for all $a \in \C A$ and $b \in \C B$. For every $s \in \nn$ and $z \in M_s(\C A \ot \C B)$ we therefore obtain that
\begin{equation}\label{eq:dervarphi}
(d_n)_s(1 \ot \varphi)_s(z) = (1 \ot \varphi)_s(d \ot 1)_s(z)
\end{equation}
and, upon consulting the definitions of the operator seminorms $\B L_D^\varphi$ and $\B L_D^{M_n(\C A)}$ from Subsection \ref{ss:mtp} and Subsection \ref{ss:matsem}, it may be concluded that
\begin{equation}\label{eq:dvarphi}
  \begin{split}
    (\B L_D^\varphi)_s(z) & = (\B L_D^{M_n(\C A)})_s\big( (1 \ot \varphi)_s(z) \big)
    = L_{D^{\op s \cd n}}\big( I_s (1 \ot \varphi)_s(z) \big)  \\
  & = \| d_{s \cd n} I_s (1 \ot \varphi)_s(z) \| =   \| (1 \ot \varphi)_s(d \ot 1)_s(z) \|.
  \end{split}
\end{equation}

  On the other hand, we have the derivation $\de : \C A \ot \C B \to \B L(H \ot_2 G)$ coming from the unital Lipschitz triple $(\C A \ot \C B, H \ot_2 G, D \hot 1)$ so that $\de(a \ot b)$ agrees with the closure of the commutator
  \[
[D \ot 1, a \ot b] : \T{dom}(D) \ot G \to H \ot_2 G 
\]
whenever $a \in \C A$ and $b \in \C B$. It follows from these observations that, upon suppressing the inclusion $\B L(H) \ot \C B \su \B L(H \ot_2 G)$, we get the identity $d \ot 1 = \de$. For every $s \in \nn$ and $z \in M_s(\C A \ot \C B)$ we therefore see that
\begin{equation}\label{eq:dhot1}
L_{(D \hot 1)^{\op s}}(z) = \| \de_s(z) \| = \| (d \ot 1)_s(z) \|. 
\end{equation}
Since $(d \ot 1)_s(z) \in M_s\big( \B L(H) \ot \C B\big) \su M_s\big( \B L(H \ot_2 G) \big)$ we obtain from the discussion in Subsection \ref{ss:mtp} that the $C^*$-norm of $(d \ot 1)_s(z)$ can be computed by the formula:
\begin{equation}\label{eq:dersuprem}
\| (d \ot 1)_s(z) \| = \sup\big\{ \| (1 \ot \varphi)_s(d \ot 1)_s(z) \| \mid \varphi \in \T{UCP}_\infty(\C B) \big\} .
\end{equation}

Combining the identities in \eqref{eq:dvarphi}, \eqref{eq:dhot1} and \eqref{eq:dersuprem} with the definition of the operator seminorm $\B L_D \ot 1$, we conclude that
\[
L_{(D \hot 1)^{\op s}}(z) = \sup\big\{ (\B L_D^\varphi)_s(z) \mid \varphi \in \T{UCP}_\infty(\C B) \big\}
= (\B L_D \ot 1)_s(z) 
\]
for all $s \in \nn$ and $z \in M_s(\C A \ot \C B)$.
\end{proof}

\begin{remark}
In the case where $(\C A,H,D)$ is a unital spectral triple and $\C B \su \B L(G)$ is a unital $*$-subalgebra, it need not be true that $(\C A \ot \C B, H \ot_2 G, D \hot 1)$ is a unital spectral triple as well. In fact, it holds that $(\C A \ot \C B, H \ot_2 G, D \hot 1)$ is a unital spectral triple if and only if $G$ is a finite dimensional Hilbert space. 
\end{remark}

\section{Matrix compact quantum metric spaces}
Throughout this section we let $\C X \su A$ be an operator system which is equipped with a Lipschitz operator seminorm $\B L = \{\B L_n\}_{n = 1}^\infty$.

For each $n \in \nn$ we identify the scalar matrices $M_n(\cc)$ with the closed subspace $M_n(\cc \cd 1_{\C X}) \su M_n(\C X)$ and equip the quotient space $M_n(\C X)/M_n(\B C)$ with the quotient norm coming from the $C^*$-norm on $M_n(\C X) \su M_n(A)$. The quotient map is denoted by $[\cd ] : M_n(\C X) \to M_n(\C X) / M_n(\B C)$.

Our aim is now to introduce the notion of a matrix compact quantum metric space. The idea is to adapt the characterization of compact quantum metric spaces in terms of finite dimensional approximations to a matricial setting. It is therefore essential to keep Theorem \ref{t:charcqms} in mind when reading the present section. The relevant adaptation is carried out by requiring that the finite dimensional approximations work simultaneously for all the matrix operator systems, $M_n(\C X)$ for $n \in \nn$. We shall present the definitions here without many explanations and then proceed almost directly with the main examples hoping that these examples will be sufficient to substantiate the new concepts. 

\begin{dfn}\label{d:cfindia}
We say that $(\C X,\B L)$ has \emph{uniformly finite diameter} when there exists a constant $C \in [0,\infty)$ such that
\[
\| [ x] \|_{M_n(\C X)/M_n(\B C)} \leq C \cd \B L_n(x)
\]
for all $n \in \nn$ and $x \in M_n(\C X)$.
\end{dfn}

\begin{dfn}\label{d:cbapprox}
Let $\ep > 0$ be a constant and let $\C Y \su B$ be an operator system. We say that a pair $(\io,\Phi)$ consisting of unital linear maps $\io,\Phi : \C X \to \C Y$ is a \emph{matricial $\ep$-approximation} of $(\C X,\B L)$ when the following holds:
\begin{enumerate}
\item $\io$ is completely isometric and $\Phi$ is completely positive; 
\item the image of $\Phi$ is a finite dimensional subspace of $\C Y$; 
\item $\| \io_n(x) - \Phi_n(x) \| \leq \ep \cd \B L_n(x)$ for all $n \in \nn$ and $x \in M_n(\C X)$.
\end{enumerate}
\end{dfn}

Combining the uniformly finite diameter condition with the existence of arbitrarily precise matricial approximations we arrive at our main definition:

\begin{dfn}\label{d:cbcqms}
We say that $(\C X,\B L)$ is a \emph{matrix compact quantum metric space} when the following holds:
\begin{enumerate}
\item $(\C X,\B L)$ has uniformly finite diameter;
\item for every $\ep > 0$ there exists a matricial $\ep$-approximation of $(\C X,\B L)$.
\end{enumerate}
\end{dfn}

Our first result states that every matrix compact quantum metric space gives rise to a compact quantum metric space by forgetting most of the Lipschitz operator seminorm $\B L = \{ \B L_n\}_{n = 1}^\infty$. The proof follows immediately by an application of Theorem \ref{t:charcqms}. We expect that the converse of the stated implication is false but we are currently unaware of a good counter example. 

\begin{prop}
If $(\C X,\B L)$ is a matrix compact quantum metric space, then $(\C X, \B L_1)$ is a compact quantum metric space.
\end{prop}

The following useful little lemma is a matrix version of the Comparison Lemma from \cite{Rie:MSA}.

\begin{lemma}\label{l:comparison}
  Let $\C Y$ be a sub-operator system of the operator system $\C X \su A$ and let $\B K = \{ \B K_n\}_{n = 1}^\infty$ be a Lipschitz operator seminorm on $\C Y$. Suppose that there exists a constant $E \in (0,\infty)$ such that
    \[
\B L_n(y) \leq E \cd \B K_n(y) \q \mbox{for all } n \in \nn \, \, \mbox{ and } \, \, \, y \in M_n(\C Y) .
\]
The following holds:
\begin{enumerate}
\item If $(\C X,\B L)$ has uniformly finite diameter, then $(\C Y,\B K)$ has uniformly finite diameter.
\item If $(\C X,\B L)$ is a matrix compact quantum metric space, then $(\C Y,\B K)$ is a matrix compact metric space.
\end{enumerate}
\end{lemma}
\begin{proof}
  The first claim regarding uniformly finite diameter is obvious. Indeed, this follows since the inclusion $i : \C Y \to \C X$ induces an isometry $[i_n] : M_n(\C Y)/M_n(\cc) \to M_n(\C X)/M_n(\cc)$ for all $n \in \nn$.

  To prove the second claim regarding matrix compact quantum metric spaces, notice that if $( \io, \Phi)$ is a matricial $\ep$-approximation of $(\C X,\B L)$ for some $\ep > 0$, then it holds that $(\io \ci i, \Phi \ci i)$ is a matricial $(\ep \cd E)$-approximation of $(\C Y,\B K)$. 
\end{proof}


The concept of a matrix compact quantum metric space makes sense in a noncommutative geometric setting by recalling that every unital spectral triple $(\C A,H,D)$ gives rise to an operator system $\C A \su \B L(H)$ equipped with a lower semicontinuous Lipschitz operator seminorm $\B L_D = \{ L_{D^{\op n}} \}_{n = 1}^\infty$, see Subsection \ref{ss:spemetspa} and Subsection \ref{ss:lipopespe}.

\begin{dfn}\label{d:cbcsps}
We say that a unital spectral triple $(\C A,H,D)$ is a \emph{matrix spectral metric space} when the pair $(\C A,\B L_D)$ is a matrix compact quantum metric space.
\end{dfn}


\subsection{Compact metric spaces}
Let us fix a compact metric space $(M,\rho)$ and consider the operator system $\T{Lip}(M,\cc) \su C(M,\cc)$ consisting of Lipschitz functions sitting inside the unital $C^*$-algebra of continuous functions.

For each $n \in \nn$ we identify the $(n \ti n)$-matrices with entries in $\T{Lip}(M,\cc)$ with the Lipschitz maps from $M$ to $M_n(\cc)$. Similarly, we identify the $(n \ti n)$-matrices with entries in $C(M,\cc)$ with the unital $C^*$-algebra $C(M,M_n(\cc))$. As in Section \ref{s:findimapp} we define the lower semicontinuous Lipschitz seminorm
\[
(\B L_\rho)_n : \T{Lip}(M,M_n(\cc)) \to [0,\infty) \q (\B L_\rho)_n(f) := \sup\big\{ \frac{\| f(p) - f(q) \|}{\rho(p,q)} \mid p \neq q \big\} .
  \]
  It can be verified that the sequence $\B L_\rho := \{ (\B L_\rho)_n\}_{n = 1}^\infty$ also satisfies the two conditions from Definition \ref{d:opesemi} so that we have a lower semicontinuous Lipschitz operator seminorm on the operator system $\T{Lip}(M,\cc) \su C(M,\cc)$. 

\begin{prop}
The pair $\big( \T{Lip}(M,\cc) , \B L_\rho \big)$ is a matrix compact quantum metric space.
\end{prop}
\begin{proof}
  Let us first verify that our data has uniformly finite diameter. Choose a point $p \in M$ and define the constant $C := \sup\big\{ \rho(p,q) \mid q \in M \big\}$. We record that $C \in [0,\infty)$ since $M$ is compact by assumption. Let now $n \in \nn$ and $f \in \T{Lip}(M,M_n(\cc))$ be given. 
We notice that
  \[
\| f(q) - f(p) \| \leq \rho(p,q) \cd (\B L_\rho)_n(f) \leq C \cd (\B L_\rho)_n(f)
\]
for all points $q \in M$ and we therefore obtain that 
\[
\begin{split}
  \| [f] \|_{ \T{Lip}(M,M_n(\cc))/M_n(\cc)}
  & \leq \| f - f(p) \| = \sup\big\{ \| f(q) - f(p) \| \mid q \in M \big\} \\
  & \leq C \cd (\B L_\rho)_n(f) .
\end{split}
  \]
  This proves that our pair has uniformly finite diameter.

  Let now $\ep > 0$ be given. We define the unital linear map $\Phi_\ep : \T{Lip}(M,\cc) \to \T{Lip}(M,\cc)$ just as we did in the proof of Lemma \ref{l:metapp} and record that $\Phi_\ep$ is in fact completely positive. The argument we already provided in the proof of Lemma \ref{l:metapp} then shows that the pair $(\T{id},\Phi_\ep)$ is a matricial $\ep$-approximation of our data $\big( \T{Lip}(M,\cc), \B L_\rho\big)$.
\end{proof}

\subsection{Ergodic actions of compact groups}\label{ss:ergodic}
Let us now focus on a compact group $G$ equipped with a length function $\ell : G \to [0,\infty)$. This means that $\ell$ is continuous and compatible with the group structure in the sense that
  \[
\ell(g h) \leq \ell(g) + \ell(h) \, \, , \, \, \, \ell(g^{-1}) = \ell(g)
\]
for all $g,h \in G$. It is moreover required that $\ell(g) = 0$ if and only if $g = e$ where $e$ is notation for the neutral element in the compact group $G$.

On top of the above data we fix a unital $C^*$-algebra $A$ and assume that $\al : G \ti A \to A$ is a strongly continuous action of $G$ by means of $*$-automorphisms of $A$. We denote the fixed point algebra by $A_G \su A$ so that
\[
A_G = \big\{ a \in A \mid \al_g(a) = a \, \, \T{for all } g \in G \big\} .
\]

We let $\C A \su A$ denote the unital $*$-subalgebra of Lipschitz elements with respect to the length function $\ell$ and the action $\al$. This unital $*$-subalgebra is defined by requiring that $a \in \C A$ if and only if $a \in A$ and the subset
\[
\big\{ \| \al_g(a) - a \| /\ell(g) \mid g \neq e \big\} \su [0,\infty)
  \]
  is bounded (or empty). We regard $\C A \su A$ as an operator system and define the lower semicontinuous Lipschitz seminorm $L : \C A \to [0,\infty)$ by putting
  \[
L(a) := \sup\big\{ \| \al_g(a) - a \| /\ell(g) \mid g \neq e \big\} .
\]
We notice that the kernel of $L$ agrees with the fixed point algebra $A_G$.

The following result is due to Rieffel, \cite{Rie:MSA}. It has subsequently been generalized by Li to ergodic actions of coamenable compact quantum groups, see \cite{Li:CQM}. 

\begin{thm}\cite[Theorem 2.3]{Rie:MSA}\label{t:rieerg}
 Suppose that the fixed point algebra $A_G \su A$ is equal to the scalars $\cc \cd 1_A$ (the action $\al$ is ergodic). Then the pair $(\C A, L)$ is a compact quantum metric space.
\end{thm}

The main goal of this subsection is to establish a matricial version of the above theorem. We record that the proof of this enhanced version is essentially the same as Rieffel's proof of \cite[Theorem 2.3]{Rie:MSA}, but for the sake of completeness we go through the relevant details. A more recent treatment, which also covers the quantum group case, can be found in \cite{Rie:CFT}.


First of all, for each $n \in \nn$ we obtain a strongly continuous action of $G$ by means of $*$-automorphisms of $M_n(A)$. Indeed, for each $g \in G$ we may consider the $*$-automorphism $(\al_g)_n : M_n(A) \to M_n(A)$ obtained by applying $\al_g : A \to A$ entry by entry. We thereby obtain the lower semicontinuous and $*$-invariant seminorm
\[
\B L_n : M_n(\C A) \to [0,\infty) \q \B L_n(a) := \sup\big\{ \| (\al_g)_n(a) - a \| /\ell(g) \mid g \neq e \big\} .
  \]
  Notice then that the sequence of seminorms $\B L := \{\B L_n\}_{n = 1}^\infty$ is in fact a lower semicontinuous Lipschitz operator seminorm on the operator system $\C A \su A$.

The notation $\T{ev}_g : C(G) \to \cc$ refers to the state which evaluates at a group element $g \in G$ and we let $\eta : C(G) \to \cc$ denote the Haar state. The strongly continuous action $\al$ can be dualized, yielding the coaction $\de : A \to A \otm C(G)$ which is characterized by the identity
\[
(1 \ot \T{ev}_g)\de(a) = \al_g(a) \q \T{for all } g \in G \T{ and } a \in A .
\]
As a consequence of the invariance properties of the Haar state we get that the composition
\[
E := (1 \ot \eta)\de : A \to A
\]
is a conditional expectation with image equal to the fixed point algebra $A_G \su A$. 
\medskip

Our proof of the matricial version of Theorem \ref{t:rieerg} relies very much on the type of estimates provided by the following lemma:


\begin{lemma}\label{l:uppdiff}
  Let $\varphi : C(G) \to \cc$ be a state. For each $n \in \nn$ and $a \in M_n(\C A)$ we have the inequality
  \[
\big\| a - (1 \ot \varphi)_n \de_n(a)  \big\| \leq \varphi( \ell) \cd \B L_n(a) .
  \]
\end{lemma}
\begin{proof}
Let $n \in \nn$ and $a \in M_n(\C A)$ be given.
  
  We consider the Banach space dual $M_n(A)'$ consisting of all the bounded linear functionals from $M_n(A)$ to $\cc$. The Banach space $M_n(A)'$ is equipped with the operator norm and we record that the $C^*$-norm of an element $b \in M_n(A)$ agrees with the supremum
  \[
\| b \| = \sup\big\{ | \mu(b)| \mid \mu \in \ov{\B B}_1(0) \big\} .
\]
Notice here that the closed unit ball $\ov{\B B}_1(0)$ sits inside $M_n(A)'$.
%

Let us consider the $*$-isomorphism $\io : M_n(A) \otm C(G) \to M_n(A \otm C(G))$ given on elementary tensors by $\io(x \ot h)_{ij} = x_{ij} \ot h$ for all $i,j \in \{1,2,\ldots,n\}$, see also \eqref{eq:matrix}. For each linear functional $\mu \in \ov{\B B}_1(0)$, we then define the continuous function $f_\mu \in C(G)$ by the formula
\[
f_\mu := \mu(a) \cd 1_{C(G)} - (\mu \ot 1) \io^{-1} \de_n(a) . 
\]
For each $g \in G$ we then notice that 
\[
\begin{split}
\big| f_\mu(g) \big| & = \big| \mu(a) - \mu\big( (1 \ot \T{ev}_g)_n \de_n(a) \big) \big| \\ 
& = \big| \mu( a ) - \mu\big( (\al_g)_n(a) \big) \big| \leq \| a - (\al_g)_n(a) \| \leq \ell(g) \cd \B L_n(a) .
\end{split}
\]
The above pointwise inequality together with standard properties of states on $C^*$-algebras now yield that
\[
\big| \varphi(f_\mu) \big| \leq \varphi\big( |f_\mu| \big) \leq \varphi( \ell ) \cd \B L_n(a) .
\]
We moreover record the identity  
\[
\varphi(f_\mu) = \mu(a) - \mu \big( (1 \ot \varphi)_n \de_n (a)\big) .
\]

A combination of the above observations yields the desired inequality: 
\[
\big\| a - (1 \ot \varphi)_n \de_n(a) \big) \big\|
= \sup\big\{ |\varphi(f_\mu)| \mid \mu \in \ov{\B B}_1(0) \big\} \leq \varphi(\ell) \cd \B L_n(a) . \qedhere
\]
\end{proof}
    
\begin{prop}\label{p:ergunifin}
Suppose that the fixed point algebra $A_G \su A$ is equal to the scalars $\cc \cd 1_A$. Then the pair $(\C A, \B L)$ has uniformly finite diameter.
\end{prop}
\begin{proof}
  Let $n \in \nn$ and $a \in M_n(\C A)$ be given. Using that $A_G$ is equal to $\cc \cd 1_A$ we know from the discussion before Lemma \ref{l:uppdiff} that $E_n(a) \in M_n(\cc \cd 1_A)$. Applying Lemma \ref{l:uppdiff} we then obtain that
\[
\| [a] \|_{M_n(\C A)/M_n(\cc)} \leq \| a - E_n(a) \| = \big\| a - (1 \ot \eta)_n \de_n(a) \big\| \leq \eta(\ell) \cd \B L_n(a). \qedhere
\]
\end{proof}

In order to construct arbitrarily precise finite dimensional matricial approximations of the pair $(\C A,\B L)$ we need to review a bit of representation theory.

Let us consider the set $\widehat{G}$ of unitary equivalence classes of irreducible unitary representations of $G$. For each $\ga \in \widehat{G}$ we choose a representative $u^\ga : G \to \C U(\cc^{d_\ga})$, where the notation $\C U(\cc^{d_\ga})$ refers to the group of unitary operators on the Hilbert space $\cc^{d_\ga}$. The representative $u^\ga$ is continuous when $\C U(\cc^{d_\ga})$ is equipped with the metric topology coming from the operator norm. We let $\{e_i\}_{i = 1}^{d_\ga}$ denote the standard basis for $\cc^{d_\ga}$ and the corresponding matrix elements are denoted by
\[
u^\ga_{ij} \in C(G) \q u_{ij}^\ga(g) := \inn{e_i, u^{\ga}(g) e_j}
\]
for all $i,j \in \{1,2,\ldots,d_\ga\}$.

For each $\ga \in \widehat{G}$ we define the bounded linear functional
\[
\rho_\ga : C(G) \to \cc \q \rho_\ga(f) := d_\ga \cd \sum_{i = 1}^{d_\ga} \eta\big( \ov{u^\ga_{ii}} \cd f\big) 
\]
and record that the Schur orthogonality relations, \cite[Proposition 5.8]{Fol:AHA}, imply that the bounded operator
\[
P_\ga : A \to A \q P_\ga := (1 \ot \rho_\ga)\de
\]
is an idempotent. The image of $P_\ga$ is called the \emph{spectral subspace} associated with $\ga$ and this subspace is denoted by $A_\ga := P_\ga(A)$. 

For each $\ga \in \widehat{G}$ and $i,j \in \{1,2,\ldots,d_\ga\}$ we shall also consider the bounded linear functional
\[
\eta_{ij}^\ga : C(G) \to \cc \q \eta_{ij}^\ga(f) := \eta( \ov{ u_{ij}^\ga} \cd f) .
\]
Another application of the Schur orthogonality relations then shows that the element
\[
(1 \ot \eta_{ij}^\ga)\de(a) 
\]
belongs to the spectral subspace $A_\ga$ whenever $a \in A$.

The main result of this subsection is now a consequence of the Peter-Weyl theorem, \cite[Theorem 5.11]{Fol:AHA}, together with a result of H\o egh-Krohn, Landstad and St\o rmer stating that ergodicity of the action $\al$ implies that all the spectral subspaces are finite dimensional, \cite[Proposition 2.1]{HLS:CEA}. In fact, for ergodic actions, it holds that $\T{dim}(A_\ga) \leq d_\ga^2$ for all $\ga \in \widehat{G}$.

\begin{thm}\label{t:ergodic}
Suppose that the fixed point algebra $A_G \su A$ is equal to the scalars $\cc \cd 1_A$. Then the pair $(\C A, \B L)$ is a matrix compact quantum metric space.
\end{thm}
\begin{proof}
  We let $\C E \su C(G)$ denote the operator system defined as the linear span of all the matrix elements of all the irreducible unitary representations of $G$. Remark that the $*$-invariance of $\C E$ follows since the contragredient representation of an irreducible unitary representation is again an irreducible unitary representation. 

  For each positive function $\psi \in \C E$ with $\eta(\psi) = 1$ we define the state
  \[
\eta_\psi : C(G) \to \cc \q \eta_\psi(f) := \eta( \psi \cd f) 
  \]
  together with the associated unital completely positive map $\Phi_\psi := (1 \ot \eta_\psi) \de : A \to A$. Since $\psi$ belongs to $\C E$ we know that $\psi$ is a finite linear combination of matrix elements and we may thus choose a finite non-empty subset $F \su \widehat{G}$ such that the image of $\Phi_\psi$ is contained in the following finite sum of spectral subspaces
  \[
\sum_{\ga \in F} A_\ga \su A .
\]
An application of \cite[Proposition 2.1]{HLS:CEA} therefore shows that the image $\Phi_\psi(A) \su A$ is finite dimensional. For each $n \in \nn$ and $a \in M_n(\C A)$ we moreover get from Lemma \ref{l:uppdiff} that
\[
\big\| a - (\Phi_\psi)_n(a) \big\| \leq \eta( \psi \cd \ell) \cd \B L_n(a) .
\]

Let now $\ep > 0$ be given. We are going to construct a matricial $\ep$-approximation of the pair $(\C A,\B L)$. This suffices to establish the result of the theorem since we already know from Proposition \ref{p:ergunifin} that our data has uniformly finite diameter.  

From the Peter-Weyl theorem we know that the operator system $\C E \su C(G)$ is dense in $C^*$-norm. Since we moreover know that the length function $\ell : G \to [0,\infty)$ is continuous and that $\ell(e) = 0$, we may choose a positive function $\psi_\ep \in \C E$ such that $\eta(\psi_\ep \cd \ell) < \ep$ and $\eta(\psi_\ep) = 1$. Letting $\io : \C A \to A$ denote the inclusion (which is certainly unital and completely isometric), it follows from the above considerations that the pair $(\io, \Phi_{\psi_\ep})$ is a matricial $\ep$-approximation.
\end{proof}

\subsection{Noncommutative tori}
Let us fix a natural number $n \in \nn$ with $n \geq 2$ together with a real skew-symmetric $(n\ti n)$-matrix $\te$. We are interested in the noncommutative torus which at the $C^*$-algebraic level is defined as the universal unital $C^*$-algebra $C(\B T_\te^n)$ generated by $n$ unitary operators $U_1,U_2,\ldots,U_n$ subject to the relations
\[
U_k U_j = e^{2 \pi i \cd \te_{jk}} U_j U_k  \q j,k \in \{1,2,\ldots,n\},
\]
see e.g. \cite{Rie:NTN}. The noncommutative geometry of the noncommutative torus is described by a unital spectral triple $\big( C^\infty(\B T_\te^n), H, D\big)$ of the same parity as $n$ and in a little while we review how this unital spectral triple is constructed, \cite{Con:GCM,CoLa:IID,CoDu:NFM}. Let us however first remark that Rieffel has applied his result on ergodic actions (see Theorem \ref{t:rieerg}) to establish that the above unital spectral triple is a spectral metric space, \cite{Rie:MSA}. We are going to show that $\big( C^\infty(\B T_\te^n), H, D\big)$ is in fact a matrix spectral metric space in the sense of Definition \ref{d:cbcqms}. This result will then be an application of Theorem \ref{t:ergodic}. In the paper \cite{Li:DCM} it is proved that theta deformations of toric spin manifolds give rise to spectral metric spaces. We believe that such general theta deformations do in fact provide examples of matrix spectral metric spaces but the details still need to be verified.

Define the strongly continuous action $\al$ of the $n$-torus on $C(\B T_\te^n)$ by the formula
\[
\al_\la(U_j) = \la_j \cd U_j \q \la = (\la_1,\la_2,\ldots,\la_n) \in \B T^n \, \, , \, \, \, j \in \{1,2,\ldots,n\} .
\]
The unital $*$-algebra of smooth elements $C^\infty(\B T_\te^n) \su C(\B T_\te^n)$ is defined as the corresponding smooth elements for the above action.
Thus, letting $\rho : \B R^n \to \B T^n$ denote the group homomorphism defined by
\[
\rho(t_1,t_2,\ldots,t_n) := (e^{i t_1}, e^{i t_2},\ldots,e^{i t_n})
\]
we have that an element $x \in C(\B T_\te^n)$ belongs to $C^\infty(\B T_\te^n)$ if and only if the map
\begin{equation}\label{eq:smooth}
f_x : \B R^n \to C(\B T_\te^n) \q f_x(t) := \al_{\rho(t)}(x)
\end{equation}
is smooth (instead of merely continuous). Apparently all the generators $U_1,U_2,\ldots,U_n$ are smooth elements and this immediately entails that $C^\infty(\B T_\te^n) \su C(\B T_\te^n)$ is dense in $C^*$-norm.

In order to construct the Hilbert space $H$ (which is part of the unital spectral triple we are describing) we shall also apply the action $\al$. It can be verified that the fixed point algebra agrees with the scalar multiples of the unit so that $\al$ is ergodic. We let $\de : C(\B T_\te^n) \to C(\B T_\te^n) \otm C(\B T^n)$ denote the coaction corresponding to $\al$ and $\eta : C(\B T^n) \to \cc$ denote the Haar state. Identifying the scalars $\cc$ with the subspace $\cc \cd 1_{C(\B T_\te^n)} \su C(\B T_\te^n)$, we obtain the faithful tracial state
\[
\tau : C(\B T_\te^n) \to \cc \q \tau(x) := (1 \ot \eta) \de . 
\]
The notation $H_\tau$ refers to the Hilbert space obtained from $\tau$ via the GNS construction and we record that the corresponding unital $*$-homomorphism $\pi_\tau : C(\B T_\te^n) \to \B L(H_\tau)$ is injective. The associated vector space inclusion of $C(\B T_\te^n)$ into $H_\tau$ is denoted by $\La : C(\B T_\te^n) \to H_\tau$. There are now two cases to consider depending on the parity of $n \geq 2$. For $n = 2m$ or $n = 2m + 1$, we consider the Hilbert space tensor product $H := H_\tau \ot_2 \cc^{2^m}$ and identify $C^\infty(\B T_\te^n)$ with the unital $*$-subalgebra of $\B L(H)$ obtained as the image of $C^\infty(\B T_\te^n)$ via the injective unital $*$-homomorphism $\pi_\tau \ot 1 : C(\B T_\te^n) \to \B L(H_\tau \ot_2 \cc^{2^m})$.

We proceed by defining the Dirac operator $D$ and the grading of the Hilbert space $H$ (for $n$ even). Let $m \in \nn$ and suppose that $n = 2m$ or $n = 2m+1$. Choose $2m+1$ selfadjoint, unitary and mutually anticommuting elements $\ga_1,\ga_2,\ldots,\ga_{2m+1}$ inside the unital $C^*$-algebra $\B L(\cc^{2^m})$. For each $j \in \{1,2,\ldots,n\}$ we consider the derivation
$\pa_j : C^\infty(\B T_\te^n) \to C(\B T_\te^n)$ defined by
\[
\pa_j(x) := \frac{\pa f_x}{\pa t_j}(0,0,\ldots,0)  \q x \in C^\infty(\B T_\te^n),
\]
where the smooth map $f_x : \B R^n \to C(\B T_\te^n)$ was introduced in \eqref{eq:smooth}. It is instructive to notice that $\pa_j(U_k) = i \cd \de_{jk} \cd U_k$ for all $k \in \{1,2,\ldots,n\}$ (with $\de_{jk}$ being the Kronecker delta and $i$ being the imaginary unit).

Upon suppressing the vector space inclusion $\La : C(\B T_\te^n) \to H_\tau$ we obtain an unbounded symmetric operator
\[
\C D_j := -i \cd \pa_j \ot \ga_j : C^\infty(\B T_\te^n) \ot \cc^{2^m} \to H_\tau \ot_2 \cc^{2^m}
\]
for every $j \in \{1,2,\ldots,n\}$. Notice that the symmetry property of $\C D_j$ follows since the faithful tracial state $\tau$ is invariant under the action $\al$, meaning that $\tau \ci \al_\la = \tau$ for all $\la \in \B T^n$. The Dirac operator $D$ is defined as the closure of the sum of these unbounded symmetric operators, so that $D := \ov{ \sum_{j = 1}^n \C D_j}$ has the algebraic tensor product $C^\infty(\B T_\te^n) \ot \cc^{2^m}$ as a core. The selfadjointness of $D$ can be proved by using that $D$ is closed and symmetric and that the Hilbert space $H$ decomposes as an orthogonal direct sum of $D$-invariant finite dimensional subspaces.  

The unital spectral triple associated to the noncommutative torus agrees with the triple $\big( C^\infty(\B T_\te^n), H, D\big)$ where we regard $H$ as a $\zz/2\zz$-graded Hilbert space with grading operator $1 \ot \ga_{2m+1}$ in the case where $n = 2m$.  

In order to show that the above unital spectral triple is a matrix spectral metric space we are going to apply the Comparison Lemma \ref{l:comparison} together with Theorem \ref{t:ergodic}. Our strategy is thus the same as the strategy applied in \cite{Rie:MSA}. However, we need to ensure that the relevant estimates are satisfied independent of the matrix size and we therefore present the main details here.

We equip the $n$-torus with the continuous length function $\ell : \B T^n \to [0,\infty)$ defined by
  \[
\ell(e^{i t_1},e^{i t_2},\ldots,e^{i t_n}) = \sqrt{ t_1^2 + t_2^2 \plp t_n^2}
\]
whenever the exponents $t_1,t_2,\ldots,t_n$ all belong to the half-open interval $(-\pi,\pi]$. Since our strongly continuous action $\al$ is ergodic we are therefore in the setting of Subsection \ref{ss:ergodic}. Letting $\T{Lip}(\B T_\te^n) \su C(\B T_\te^n)$ denote the corresponding unital $*$-subalgebra of Lipschitz elements (coming from the action $\al$ and the length function $\ell$), we obtain from Theorem \ref{t:ergodic} that the pair $\big( \T{Lip}(\B T_\te^n),\B L\big)$ is a matrix compact quantum metric space. Recall in this respect that
\[
\B L_s(x) = \sup\big\{ \| (\al_\la)_s(x) - x \| / \ell(\la) \mid \la \in \B T^n \sem \{ (1,1,\ldots,1) \} \big\}
\]
for all $s \in \nn$ and $x \in M_s\big( \T{Lip}(\B T_\te^n) \big)$.

\begin{lemma}\label{l:smoolip}
We have the inequality
  \[
\B L_s(x) \leq \sqrt{n} \cd L_{D^{\op s}}(x)
  \]
  for all $s \in \nn$ and $x \in M_s\big( C^\infty(\B T_\te^n) \big)$. In particular, it holds that $C^\infty(\B T_\te^n)$ is a sub-operator system of $\T{Lip}(\B T_\te^n)$.
\end{lemma}
\begin{proof}
Choose an $m \in \nn$ such that $n = 2m$ or $n = 2m + 1$. The notation $d : C^\infty(\B T_\te^n) \to \B L(H_\tau \ot_2 \cc^{2^m})$ refers to the derivation coming from the unital spectral triple $\big( C^\infty(\B T_\te^n), H, D)$. This derivation is given more explicitly by the formula $d(a) = -i \cd \sum_{j = 1}^n \pa_j(a) \ot \ga_j$ for all $a \in C^\infty(\B T_\te^n)$. We are also interested in the derivation $\pa_k \ot 1 : C^\infty(\B T_\te^n) \to \B L(H_\tau \ot_2 \cc^{2^m})$ defined by the formula $(\pa_k \ot 1)(a) := \pa_k(a) \ot 1$ for every $k \in \{1,2,\ldots,n\}$.

  Let $s \in \nn$ and $x \in M_s\big( C^\infty(\B T_\te^n)\big)$ be given. 

  For each $k \in \{1,2,\ldots,n\}$ we consider the matrix of bounded operators $(1 \ot \ga_k)^{\op s} \in M_s\big( \B L(H_\tau \ot_2 \cc^{2^m}) \big)$ (having $1 \ot \ga_k$ repeated $s$ times on the diagonal and zeroes elsewhere). Using the anticommutation relations for our Clifford matrices $\ga_1,\ga_2,\ldots,\ga_n$ we then get that
  \[
\frac{1}{2} (1 \ot \ga_k)^{\op s} \cd d_s(x) + \frac{1}{2} d_s(x) \cd (1 \ot \ga_k)^{\op s} = -i \cd (\pa_k \ot 1)_s(x).
\]
This identity entails that
\begin{equation}\label{eq:partdirac}
\| (\pa_k)_s(x) \| = \| (\pa_k \ot 1)_s(x) \| \leq \| d_s(x) \| = \| L_{D^{\op s}}(x) \|.
\end{equation}

We now define the smooth map $f_x : \B R^n \to M_s\big( C(\B T_\te^n) \big)$ by putting
\[
f_x(t) = (\al_{\rho(t)})_s(x) \q \T{for all } t \in \B R^n .
\]
It then holds that $(\pa_k)_s(x) = \frac{\pa f_x}{\pa t_k}(0,0,\ldots,0)$ for all $k \in \{1,2,\ldots,n\}$.

Let us fix an element $\mu \in \B T^n$ and choose a point $p = (p_1,p_2,\ldots,p_n) \in (-\pi,\pi]^n$ such that $\rho(p) = \mu$. In particular we have that $\ell(\mu) = \sqrt{ \sum_{j = 1}^n p_j^2}$. Define the smooth path $c : \B R \to \B R^n$ by putting $c(r) := r \cd p$. A few straightforward computations then establish the identities
\[
(f_x \ci c)'(r) = ( \al_{ \rho(c(r))})_s\big( (f_x \ci c)'(0) \big) = ( \al_{ \rho(c(r))})_s\big( \sum_{j = 1}^n p_j \cd (\pa_j)_s(x) \big)
\]
for all $r \in \B R$. An application of the fundamental theorem of calculus together with \eqref{eq:partdirac} now shows that
\begin{equation}\label{eq:lengthdirac}
  \begin{split}
\| (\al_{\mu})_s(x) - x \| & = \big\| \int_0^1 (f_x \ci c)'(r) \, dr \big\| \leq \int_0^1 \big\| \sum_{j = 1}^n p_j \cd (\pa_j)_s(x) \big\| \, dr \\
& \leq \sum_{j = 1}^n |p_j| \cd L_{D^{\op s}}(x) \leq \sqrt{n} \cd \ell( \mu) \cd L_{D^{\op s}}(x) .
\end{split}
\end{equation}

We conclude from the inequality in \eqref{eq:lengthdirac} that $\B L_s(x) \leq \sqrt{n} \cd L_{D^{\op s}}(x)$ and the lemma is therefore proved.
%
%
%
\end{proof}

As mentioned earlier, we now obtain the main result of this subsection as a consequence of Theorem \ref{t:ergodic}, the estimates in Lemma \ref{l:smoolip} and the Comparison Lemma \ref{l:comparison}.

\begin{thm}
The unital spectral triple $\big( C^\infty(\B T_\te^n), H, D\big)$ is a matrix spectral metric space.
\end{thm}

\subsection{The Podle\'s sphere}
Let us fix a $q \in (0,1]$ and consider the $C^*$-algebraic version of the quantum group $SU_q(2)$, \cite{Wor:CMP,Wor:TGN}. This compact quantum group is defined as the universal unital $C^*$-algebra $C(SU_q(2))$ with two generators $a$ and $b$ subject to the relations
\[
\begin{split}
& ba = q ab \q b^* a = q a b^* \q bb^* = b^* b \\
& a^* a + q^2 bb^* = 1_{C(SU_q(2))} = aa^* + bb^* .
\end{split}
\]
The corresponding coproduct $\De : C(SU_q(2)) \to C(SU_q(2)) \otm C(SU_q(2))$ is given by the formulae
\[
\De(a) := a \ot a  - q b^* \ot b  \q \De(b) := b \ot a  + a^* \ot b  
\]
and the counit $\epsilon : C(SU_q(2)) \to \cc$ is determined by putting $\epsilon(a) = 1$ and $\epsilon(b) = 0$. Inside the unital $C^*$-algebra $C(SU_q(2))$ we identify the $C^*$-algebraic version of the (standard) Podle\'s sphere $S_q^2$, defined as the smallest unital $C^*$-subalgebra $C(S_q^2) \su C(SU_q(2))$ containing the two elements $A := bb^*$ and $B := a b^*$, \cite{Pod:QS}.

We introduce the coordinate algebra for quantum $SU(2)$, denoted by $\C O(SU_q(2))$, as the smallest unital $*$-subalgebra of $C(SU_q(2))$ such that $a$ and $b$ belong to $\C O(SU_q(2))$. We emphasize that $\C O(SU_q(2))$ is a unital Hopf $*$-algebra with coproduct and counit induced by the corresponding unital $*$-homomorphisms at the $C^*$-algebraic level, see e.g. \cite[Chapter 4]{KlSc:QGR}. The antipode $S : \C O(SU_q(2)) \to \C O(SU_q(2))$ is given on generators by the formulae $S(a) = a^*$ and $S(b) = -q^{-1}b$. The coordinate algebra for the Podle\'s sphere $\C O(S_q^2)$ is defined as the unital $*$-subalgebra of $C(S_q^2)$ generated by $A$ and $B$. The coproduct $\De : \C O(SU_q(2)) \to \C O(SU_q(2)) \ot \C O(SU_q(2))$ induces a coaction of quantum $SU(2)$ on the Podle\'s sphere, which we denote by $\de : \C O(S_q^2) \to \C O(SU_q(2)) \ot \C O(S_q^2)$.

We are interested in the quantum metric information pertaining to the Podle\'s sphere. One way of gaining access to this information is to consider the noncommutative geometry described by the Dabrowski-Sitarz spectral triple, $\big( \C O(S_q^2), H_+ \op H_-, D_q\big)$, which was introduced in \cite{DaSi:DSP}. It has been proved in \cite{AgKa:PSM} that $\big( \C O(S_q^2), H_+ \op H_-, D_q\big)$ is a spectral metric space and we shall now improve this result by showing that the Dabrowski-Sitarz spectral triple is in fact a matrix spectral metric space. Notice that the paper \cite{AgKa:PSM} treats the substantially larger unital $*$-algebra of Lipschitz elements $\T{Lip}_{D_q}(S_q^2)$ instead of the coordinate algebra $\C O(S_q^2)$. We believe that $\big( \T{Lip}_{D_q}(S_q^2), H_+ \op H_-, D_q\big)$ is also a matrix spectral metric space but, for the time being, we leave out the details. 

Instead of reviewing the construction of the Dabrowski-Sitarz spectral triple we recall from \cite[Proposition 3.1]{NeTu:LIQ} that commutators with the Dirac operator $D_q$ give rise to two derivations
\[
\pa_1 \T{ and } \pa_2 : \C O(S_q^2) \to \C O(SU_q(2))
\]
satisfying that $\pa_1(x^*) = - \pa_2(x)^*$ for all $x \in \C O(S_q^2)$. Moreover, these two derivations are described on generators by the formulae
\[
\begin{split}
& \pa_1(A) = -b^* a^* \q \pa_1(B) = (b^*)^2 \q \pa_1(B^*) = -q^{-1} (a^*)^2 \\
& \pa_2(A) = ab \q \pa_2(B) = q^{-1} a^2 \q \pa_2(B^*) = -b^2 .
\end{split}
\]
We may then express the lower semicontinuous Lipschitz operator seminorm $\B L_{D_q} = \{ L_{D_q^{\op s}} \}_{s = 1}^\infty$ coming from the Dabrowski-Sitarz spectral triple in terms of the above derivations. Indeed, for each $s \in \nn$ and $x \in M_s( \C O(S_q^2))$ we have that
\[
L_{D_q^{\op s}}(x) := \max\big\{ \big\| (\pa_1)_s(x) \big\| , \big\| (\pa_2)_s(x) \big\| \big\}.
\]
%

Here below we consider the coordinate algebras $\C O(S_q^2) \su C(S_q^2)$ and $\C O(SU_q(2)) \su C(SU_q(2))$ as operator systems. Notice that the unital complete isometry $\io : M_s\big( \C O(SU_q(2))\big) \ot \C O(S_q^2) \to M_s\big( \C O(SU_q(2)) \ot \C O(S_q^2)\big)$ appearing in the statement of the lemma was introduced in \eqref{eq:matrix}. 

\begin{lemma}\label{l:contraction}
  Let $s \in \nn$ and let $\mu : M_s\big( \C O(SU_q(2)) \big) \to \cc$ be a linear contraction. We have the inequality
  \[
L_{D_q}\big( (\mu \ot 1) \io^{-1} \de_s(x) \big) \leq L_{D_q^{\op s}}(x) \q \mbox{for all } x \in M_s\big(\C O(S_q^2)\big).
\]
\end{lemma}
\begin{proof}
  We start out by noting the identities $(1 \ot \pa_1) \de = \De \pa_1$ and $(1 \ot \pa_2) \de = \De \pa_2$, see e.g. \cite[Lemma 4.1]{AKK:PSC}. This entails that 
  \[
\pa_1 (\mu \ot 1) \io^{-1} \de_s = (\mu \ot 1) \io^{-1} (1 \ot \pa_1)_s \de_s = (\mu \ot 1) \io^{-1} \De_s (\pa_1)_s 
\]
and similarly that $\pa_2 (\mu \ot 1) \io^{-1} \de_s = (\mu \ot 1) \io^{-1} \De_s (\pa_2)_s$. Since both $(\mu \ot 1)\io^{-1}$ and $\De_s$ have operator norm less than or equal to one we have proved the present lemma.
\end{proof}

Let $n \in \nn_0$. Before stating our next result we review the definition of the quantum Berezin transform $\be_n : \C O(S_q^2) \to \C O(S_q^2)$ following \cite{AKK:PSC}, but we also refer to \cite{INT:PBD} and \cite{Sain:Thesis} for more general treatments of quantum Berezin transforms. Let $\eta : C(SU_q(2)) \to \cc$ denote the Haar state and define the state $\eta_n : \C O(S_q^2) \to \cc$ by putting $\eta_n(x) := \eta( (a^*)^n x a^n) \cd \sum_{j = 0}^n q^{2j}$ for all $x \in \C O(S_q^2)$. The quantum Berezin transform is then given as the composition
\[
\be_n := (1 \ot \eta_n) \de .
\]
A priori $\be_n$ takes values in $\C O(SU_q(2))$ but it can be verified that the above composition factorizes through $\C O(S_q^2)$. The image of the quantum Berezin transform is computed explicitly in \cite[Lemma 3.7]{AKK:PSC} and this image turns out to be finite dimensional. We moreover remark that $\be_n$ is unital completely positive being the composition of a unital $*$-homomorphism (which extends to $C(S_q^2)$) and a slice map coming from the state $\eta_n$.  

Let $\rho_q : S\big( \C O(S_q^2)\big) \ti S\big( \C O(S_q^2)\big) \to [0,\infty)$ denote the Monge-Kantorovic metric on the state space of $\C O(S_q^2)$ coming from the Dabrowski-Sitarz spectral triple. Let us recall that
  \[
\rho_q( \varphi,\psi) = \sup\big\{ | \varphi(x) - \psi(x)| \mid x \in \C O(S_q^2) \, , \, \, L_{D_q}(x) \leq 1 \big\}
\]
whenever $\varphi$ and $\psi$ are states on the coordinate algebra $\C O(S_q^2)$ (viewed as an operator system inside $C(S_q^2)$). Notice that $\rho_q$ metrizes the weak-$*$-topology on the state space since we already know from \cite[Theorem 8.3]{AgKa:PSM} that $\big( \C O(S_q^2), L_{D_q} \big)$ is a compact quantum metric space. In particular, we know that the distance between two states is never equal to infinity.

The next proposition shows that the identity operator $\T{id} : \C O(S_q^2) \to \C O(S_q^2)$ and the quantum Berezin transform $\be_n : \C O(S_q^2) \to \C O(S_q^2)$ yield a matricial $\ep$-approximation of the pair $\big( \C O(S_q^2), L_{D_q} \big)$ for $\ep = \rho_q(\epsilon,\eta_n)$ (we are here viewing the restriction of the counit as a state on $\C O(S_q^2)$). 

\begin{prop}\label{p:berezin}
It holds that
  \[
\| x - (\be_n)_s(x) \| \leq \rho_q( \epsilon,\eta_n) \cd L_{D_q^{\op s}}(x)
\]
for all $s \in \nn$ and $x \in M_s\big( \C O(S_q^2)\big)$.
\end{prop}
\begin{proof}
Let $s \in \nn$ and $x \in M_s\big( \C O(S_q^2)\big)$ be given.
 
For every linear contraction $\mu : M_s\big( \C O(SU_q(2))\big) \to \cc$ we compute that
\[
\begin{split}
\mu\big(  x - (\be_n)_s(x) \big) & = \mu (1 \ot (\epsilon - \eta_n))_s \de_s(x) = \mu (1 \ot (\epsilon - \eta_n) ) \io^{-1} \de_s(x) \\
& = (\epsilon - \eta_n) (\mu \ot 1) \io^{-1} \de_s(x) .
\end{split}
\]
But we know from Lemma \ref{l:contraction} that $L_{D_q}\big( (\mu \ot 1) \io^{-1} \de_s(x) \big) \leq L_{D_q^{\op s}}(x)$ and we therefore get that
\[
\big| \mu\big(  x - (\be_n)_s(x) \big) \big| \leq \rho_q( \epsilon, \eta_n) \cd L_{D_q^{\op s}}(x) .
\]

By taking the supremum over all linear contractions from $M_s\big( \C O(SU_q(2))\big)$ to $\cc$ we obtain the result of the proposition.
\end{proof}

\begin{thm}
The unital spectral triple $\big( \C O(S_q^2), H_+ \op H_-, D_q \big)$ is a matrix spectral metric space.
\end{thm}
\begin{proof}
  We need to verify that $\big( \C O(S_q^2), \B L_{D_q} \big)$ has uniformly finite diameter and we need to construct a matricial $\ep$-approximation for every $\ep > 0$.

  The fact that our pair has uniformly finite diameter follows from Proposition \ref{p:berezin} by noting that $\be_0(x) = (1 \ot \eta)\De(x) = \eta(x) \cd 1_{\C O(S_q^2)}$ for all $x \in \C O(S_q^2)$. Indeed, we then get that
  \[
\| [x] \|_{M_s(\C O(S_q^2))/M_s(\cc)} \leq \| x - \eta_s(x)  \| \leq \rho_q(\epsilon,\eta) \cd L_{D_q^{\op s}}(x)
\]
for all $s \in \nn$ and $x \in M_s\big( \C O(S_q^2) \big)$.

Let now $\ep > 0$ be given. From \cite[Proposition 4.4]{AKK:PSC} we get that $\lim_{n \to \infty} \rho_q( \epsilon, \eta_n) = 0$ and we may thus choose an $N \in \nn_0$ such that $\rho_q(\epsilon,\eta_N) \leq \ep$. Since the quantum Berezin transform $\be_N : \C O(S_q^2) \to \C O(S_q^2)$ is unital completely positive and has finite dimensional image (\cite[Lemma 3.6]{AKK:PSC}) we obtain from Proposition \ref{p:berezin} that the pair $\big( \T{id}, \be_N \big)$ is a matricial $\ep$-approximation of $\big( \C O(S_q^2), \B L_{D_q} \big)$.  
\end{proof}

\section{External products of compact quantum metric spaces}
Throughout this section we let $\C X \su A$ and $\C Y \su B$ be operator systems equipped with Lipschitz operator seminorms $\B L = \{\B L_s\}_{s = 1}^\infty$ and $\B K = \{\B K_s\}_{s = 1}^\infty$, respectively.  

The aim of this section is to show that the operator system $\C X \ot \C Y \su A \otm B$ can be given the structure of a matrix compact quantum metric space once we know that $(\C X,\B L)$ and $(\C Y,\B K)$ are matrix compact quantum metric spaces. One may for example apply the maximum of the two Lipschitz operator seminorms $\B L \ot 1$ and $1 \ot \B K$ from Lemma \ref{l:tenslip}, but there are other choices available. 
%

Here below, the operator system $\C Y$ is often (and tacitly) identified with the sub-operator system $\cc 1_{\C X} \ot \C Y$ of $\C X \ot \C Y$ via the unital completely isometric map 
\[
i_{\C Y} : \C Y \to \C X \ot \C Y \q y \mapsto 1_{\C X} \ot y .
\]

%

\begin{lemma}\label{l:whyfac}
  Let $\psi : \C X \to \cc$ be a state. Suppose that the pair $(\C X,\B L)$ has uniformly finite diameter and let $C \in [0,\infty)$ be a constant such that $\| [x] \|_{M_s(\C X)/M_s(\cc)} \leq C \cd \B L_s(x)$ for all $s \in \nn$ and $x \in M_s(\C X)$. Then it holds that
\[
\| z - (\psi \ot 1)_s(z) \| \leq 2 C \cd (\B L \ot 1)_s(z)
\q \mbox{for all } s \in \nn \, \, \mbox{ and } \, \, \, z \in M_s(\C X \ot \C Y) .
\]
\end{lemma}
\begin{proof}
Let $s \in \nn$ and $z \in M_s(\C X \ot \C Y)$ be given. For $n \in \nn$ we consider an arbitrary unital completely positive map $\varphi : \C Y \to M_n(\cc)$. Our task is then to show that
\[
\| (1 \ot \varphi)_s(z) - \varphi_s(\psi \ot 1)_s(z) \| \leq 2 C \cd (\B L \ot 1)_s(z), 
\]
where we recall that the operators on the left hand side lives in $M_s( M_n(\C X))$, see Subsection \ref{ss:mtp}. We first record that $\varphi_s(\psi \ot 1)_s(z) = (\psi_n)_s (1 \ot \varphi)_s(z)$ and we therefore obtain that
\[
\| (1 \ot \varphi)_s(z) - \varphi_s(\psi \ot 1)_s(z) \| \leq 2 \cd \big\| \big[ (1 \ot \varphi)_s(z) \big] \big\|_{M_s(M_n(\C X))/M_s(M_n(\cc))} .
\]
Identifying $M_s( M_n(\C X))$ with $M_{s \cd n}(\C X)$ via the completely isometric isomorphism $I_s$ (see \eqref{eq:matsemi}) and applying our assumption regarding uniformly finite diameter we then see that
\[
\begin{split}
  2 \cd \big\| \big[ (1 \ot \varphi)_s(z) \big] \big\|_{M_s(M_n(\C X))/M_s(M_n(\cc))}
  & \leq 2 C \cd \B L_{s \cd n}\big(  I_s\big( (1 \ot \varphi)_s(z) \big) \big) \\
  & = 2 C \cd \B L_s^\varphi(z) \leq 2C \cd (\B L \ot 1)_s(z) .
  \end{split}
\]
This ends the proof of the present lemma.
\end{proof}

\begin{prop}\label{p:tensfindia}
Let $\B M$ be a Lipschitz operator seminorm on $\C X \ot \C Y \su A \otm B$. Suppose that $(\C X,\B L)$ and $(\C Y,\B K)$ both have uniformly finite diameter and that there exists a constant $D \geq 0$ such that
\[
(\B L \ot 1)_s( z) , (1 \ot \B K)_s(z) \leq D \cd \B M_s(z) \q \mbox{for all } s \in \nn \, \, \mbox{ and } \, \, \, z \in M_s(\C X \ot \C Y).
\]
Then $(\C X \ot \C Y, \B M)$ has uniformly finite diameter. 
\end{prop}
\begin{proof}
Since $(\C X,\B L)$ and $(\C Y,\B K)$ both have uniformly finite diameter we may choose constants $C_L \geq 0$ and $C_K \geq 0$ such that
\[
\| [x] \|_{M_s(\C X) / M_s(\cc) } \leq C_L \cd \B L_s(x) \, \, \T{ and } \, \, \,
\| [y] \|_{M_s(\C Y)/ M_s(\cc)} \leq C_K \cd \B K_s(y)
\]
for all $s \in \nn$, $x \in M_s(\C X)$ and $y \in M_s(\C Y)$. 

Let $\psi : \C X \to \cc$ and $\varphi : \C Y \to \cc$ be states. For each $s \in \nn$ and $z \in M_s(\C X \ot \C Y)$ we apply Lemma \ref{l:whyfac} and estimate that
\[
\begin{split}
& \| [z] \|_{ M_s(\C X \ot \C Y) / M_s(\B C) } \leq \| z - \varphi_s(\psi \ot 1)_s(z) \| \\
& \q \leq \| z - (\psi \ot 1)_s(z) \| + \| (\psi \ot 1)_s(z) - \varphi_s(\psi \ot 1)_s(z)  \| \\
& \q \leq 2 C_L \cd (\B L \ot 1)_s(z) + 2 \cd \big\| \big[ (\psi \ot 1)_s( z ) \big] \big\|_{M_s(\C Y)/M_s(\B C)}  \\
& \q \leq 2 C_L \cd (\B L \ot 1)_s(z) + 2 C_K \cd \B K_s\big( (\psi \ot 1)_s( z) \big) \\
& \q \leq 2 C_L \cd (\B L \ot 1)_s(z) + 2 C_K \cd (1 \ot \B K)_s(z) \leq 2 (C_L + C_K) D \cd \B M_s( z) .
\end{split}
\]
These inequalities prove the proposition. 
\end{proof}

We are now ready to investigate finite dimensional matricial approximations of the operator system $\C X \ot \C Y \su A \otm B$ equipped with the maximum of the two Lipschitz operator seminorms $\B L \ot 1$ and $1 \ot \B K$.

%
%

\begin{prop}\label{p:tensapprox}
Let $\ep_{\C X} > 0$ and $\ep_{\C Y} > 0$ be constants and suppose that $(\io_{\C X},\Phi_{\C X})$ is a matricial $\ep_{\C X}$-approximation of $(\C X,\B L)$ and that $(\io_{\C Y},\Phi_{\C Y})$ is a matricial $\ep_{\C Y}$-approximations of $(\C Y,\B K)$. Suppose moreover that $\B M$ is a Lipschitz operator seminorm on the algebraic tensor product $\C X \ot \C Y \su A \otm B$ and that $D > 0$ is a constant such that
\[
(\B L \ot 1)_s(z) , (1 \ot \B K)_s(z) \leq D \cd \B M_s(z) \q \mbox{for all } s \in \nn \, \, \mbox{ and } \, \, \, z \in M_s(\C X \ot \C Y) .
\]
Then $(\io_{\C X} \ot \io_{\C Y}, \Phi_{\C X} \ot \Phi_{\C Y})$ is a matricial $(D\ep_{\C X} + D\ep_{\C Y})$-approximation of $(\C X \ot \C Y,\B M)$.
\end{prop}
\begin{proof}
  We prove condition $(3)$ from Definition \ref{d:cbapprox} since the other two conditions are straightforward to establish. Let $s \in \nn$ and $z \in M_s(\C X \ot \C Y)$ be given. For $m \in \nn$ and an arbitrary unital completely positive map $\psi : \C X \to M_m(\cc)$ we first record that
  \[
(\psi \ot 1)_s(1 \ot (\io_{\C Y} - \Phi_{\C Y}))_s(z) = \big( (\io_{\C Y} - \Phi_{\C Y} )_m \big)_s (\psi \ot 1)_s(z) .
  \]
  Hence, applying the completely isometric isomorphism $I_s$ between $M_s(M_m(\C Y))$ and $M_{s \cd m}(\C Y)$ (see \eqref{eq:matsemi}), we obtain that
  \begin{equation}\label{eq:psitens}
  \begin{split}
    \big\| (\psi \ot 1)_s(1 \ot (\io_{\C Y} - \Phi_{\C Y}))_s(z) \big\|
    & = \big\| (\io_{\C Y} - \Phi_{\C Y})_{s \cd m} I_s (\psi \ot 1)_s(z) \big\| \\
    & \leq \ep_{\C Y} \cd \B K_{s \cd m}\big( I_s (\psi \ot 1)_s(z) \big) \\
    & = \ep_{\C Y} \cd \B K_s^\psi(z) \leq \ep_{\C Y} \cd (1 \ot \B K)_s(z) .
  \end{split}
\end{equation}

We conclude from \eqref{eq:psitens} and the definition of the $C^*$-norm on $M_s(\C X \ot \C Y) \su M_s(A \otm B)$ that
\[
\big\| (1 \ot (\io_{\C Y} - \Phi_{\C Y}))_s(z) \big\| \leq \ep_{\C Y} \cd (1 \ot \B K)_s(z) .
\]
A similar computation shows that $\big\| ((\io_{\C X} - \Phi_{\C X}) \ot 1)_s(z) \big\| \leq \ep_{\C X} \cd (\B L \ot 1)_s(z)$. We may then estimate as follows:
\[
\begin{split}
& \big\| (\io_{\C X} \ot \io_{\C Y})_s(z) - (\Phi_{\C X}  \ot \Phi_{\C Y} )_s(z) \big\| \\
  & \q \leq \big\| \big( (\io_{\C X} - \Phi_{\C X} ) \ot \io_{\C Y}  \big)_s(z) \big\|
  + \big\| \big(\Phi_{\C X} \ot (\io_{\C Y} - \Phi_{\C Y} )\big)_s(z) \big\| \\
& \q \leq \big\| \big( (\io_{\C X} - \Phi_{\C X}) \ot 1\big)_s(z) \big\| + \big\| \big(1 \ot (\io_{\C Y} - \Phi_{\C Y} )_s\big)(z) \big\|  \\
& \q \leq \ep_{\C X} \cd (\B L \ot 1)_s(z)  + \ep_{\C Y} \cd (1 \ot \B K)_s(z) 
\leq (D \ep_{\C X} + D \ep_{\C Y}) \cd \B M_s(z) .
\end{split}
\]
This proves the proposition. 
\end{proof}

The next theorem, which can be seen as the main result of this paper, can now be proved by applying Proposition \ref{p:tensfindia} and Proposition \ref{p:tensapprox}. 

\begin{thm}\label{t:extquamet}
Suppose that $(\C X,\B L)$ and $(\C Y,\B K)$ are matrix compact quantum metric spaces and suppose that $\B M$ is a Lipschitz operator seminorm on $\C X \ot \C Y$ and that $D \geq 0$ is a constant such that
\[
(\B L \ot 1)_s(z) , (1 \ot \B K)_s(z) \leq D \cd \B M_s(z) \q \mbox{for all } s \in \nn \, \, \mbox{ and } \, \, \, z \in M_s(\C X \ot \C Y) .
\]
Then $(\C X \ot \C Y,\B M)$ is a matrix compact quantum metric space. 
\end{thm}

\begin{remark}
Under the general assumptions applied in this section, it always holds that the sequence of seminorms $\B M = \{\B M_s\}_{s = 1}^\infty$ defined by
\[
\B M_s(z) := \max\{ (\B L \ot 1)_s(z), (1 \ot \B K)_s(z) \}
\]
for all $s \in \nn$ and $z \in M_s(\C X \ot \C Y)$ is a Lipschitz operator seminorm on $\C X \ot \C Y$. This is a consequence of Lemma \ref{l:tenslip}. It then follows from Theorem \ref{t:extquamet} that if $(\C X,\B L)$ and $(\C Y,\B K)$ are matrix compact quantum metric spaces, then this holds true for $(\C X \ot \C Y, \B M)$ as well. As we shall see in the next section, there could be other interesting Lipschitz operator seminorms on $\C X \ot \C Y$ satisfying the condition in Theorem \ref{t:extquamet}.
\end{remark}

We end this section by stating a slightly stronger result than the above Theorem \ref{t:extquamet}. The proof of this result is almost identical to the proof of Theorem \ref{t:extquamet} and there is therefore no need to go through the details here. We recall the definition of the Lipschitz operator seminorms $\B L \hot 1$ and $1 \hot \B K$ from the paragraph before Lemma \ref{l:tensmax}. 

\begin{thm}
Suppose that $(\C X,\B L)$ and $(\C Y,\B K)$ are matrix compact quantum metric spaces and suppose that $\B M$ is a Lipschitz operator seminorm on the sub-operator system $\C D(\B L \hot 1) \cap \C D(1 \hot \B K) \su X \otm Y$ and that $D \geq 0$ is a constant such that
\[
(\B L \hot 1)_s(z) , (1 \hot \B K)_s(z) \leq D \cd \B M_s(z) 
\]
for all $s \in \nn$ and $z \in M_s\big(\C D(\B L \hot 1) \cap \C D(1 \hot \B K)\big)$. Then the pair $\big(\C D(\B L \hot 1) \cap \C D(1 \hot \B K),\B M\big)$ is a matrix compact quantum metric space. 
\end{thm}



\section{External products of spectral metric spaces}
The aim of this section is to apply our main Theorem \ref{t:extquamet} to show that the external product of two matrix spectral metric spaces is again a matrix spectral metric space. We let $(\C A_1,H_1,D_1)$ and $(\C A_2,H_2,D_2)$ be unital spectral triples and, in the case where $(\C A_i,H_i,D_i)$ is even for some $i \in \{1,2\}$, we denote the corresponding $\zz/2\zz$-grading operator on the separable Hilbert space $H_i$ by $\ga_i : H_i \to H_i$.

Let us first review the construction of the external product of our two unital spectral triples. The main reference for this construction is Baaj and Julg, \cite{BaJu:TBK}, but we remark that Baaj and Julg work in the much more general context of unbounded Kasparov modules. The external product of unital spectral triples recovers the external product in analytic $K$-homology via the bounded transform and, more generally, this relationship still holds between the external product of unbounded Kasparov modules and the external product in Kasparov's $KK$-theory, \cite{BaJu:TBK,Kuc:PUM,HiRo:AKH,Kas:OFE}.  

There are four different cases of the external product corresponding to the different combinations of parities of $(\C A_i,H_i,D_i)$, $i \in \{1,2\}$. In all four cases the external product takes the form $( \C A_1 \ot \C A_2, H, D_1 \ti D_2 )$, where the Hilbert space $H$ depends on the combination of parities and the selfadjoint unbounded operator $D_1 \ti D_2 : \T{dom}(D_1 \ti D_2) \to H$ is referred to as the \emph{unbounded product operator}.  

We specify the other ingredients in the external product here below, arranged after parity combinations. Unless explicitly mentioned we shall view the algebraic tensor product $\C A_1 \ot \C A_2$ as a unital $*$-subalgebra of the bounded operators on the Hilbert space tensor product $H_1 \ot_2 H_2$.

{\bf Even times even:} We define $H := H_1 \ot_2 H_2$ and equip this Hilbert space with the $\zz/2\zz$-grading operator $\ga_1 \ot \ga_2 : H_1 \ot_2 H_2 \to H_1 \ot_2 H_2$. The unbounded product operator $D_1 \ti D_2$ then agrees with the closure of the odd symmetric unbounded operator
\[
D_1 \ot 1 + \ga_1 \ot D_2 : \T{dom}(D_1) \ot \T{dom}(D_2) \to H_1 \ot_2 H_2 .
\]

{\bf Odd times odd:} We put $H := (H_1 \ot_2 H_2)^{\op 2}$ and equip this direct sum of Hilbert spaces with the $\zz/2\zz$-grading operator $\ga := \ma{cc}{1 & 0 \\ 0 & -1} : H \to H$. The algebraic tensor product $\C A_1 \ot \C A_2$ is then viewed as a unital $*$-subalgebra of the bounded operators on $H$ via the diagonal representation $x \mapsto \ma{cc}{x & 0 \\ 0 & x}$. The unbounded product operator $D_1 \ti D_2$ is defined as the closure of the odd symmetric unbounded operator
\[
\ma{cc}{0 & D_1 \ot 1 + i \ot D_2 \\ D_1 \ot 1 - i \ot D_2 & 0} : \big( \T{dom}(D_1) \ot \T{dom}(D_2) \big)^{\op 2} \to (H_1 \ot_2 H_2)^{\op 2} .
\]

{\bf Odd times even and even times odd:} In both of these two cases, the relevant Hilbert space agrees with the Hilbert space tensor product $H := H_1 \ot_2 H_2$ without any grading. In the odd times even case, we define the unbounded product operator $D_1 \ti D_2$ as the closure of the symmetric unbounded operator
\[
D_1 \ot \ga_2 + 1 \ot D_2 : \T{dom}(D_1) \ot \T{dom}(D_2) \to H_1 \ot_2 H_2 .
\]
In the even times odd case, the unbounded product operator $D_1 \ti D_2$ agrees with the closure of the symmetric unbounded operator
\[
D_1 \ot 1 + \ga_1 \ot D_2 : \T{dom}(D_1) \ot \T{dom}(D_2) \to H_1 \ot_2 H_2 .
\]



%
%
%
%

\begin{thm}
If $(\C A_1,H_1,D_1)$ and $(\C A_2,H_2,D_2)$ are matrix spectral metric spaces, then the external product $(\C A_1 \ot \C A_2, H, D_1 \ti D_2)$ is a matrix spectral metric space.
\end{thm}
\begin{proof}
Let $s \in \nn$ and $z \in M_s(\C A_1 \ot \C A_2)$ be given. By Proposition \ref{p:stable} and Theorem \ref{t:extquamet} it suffices to show that
\[
L_{(D_1 \hot 1)^{\op s}}(z) \, , \, \, L_{(1 \hot D_2)^{\op s}}(z) \leq L_{(D_1 \ti D_2)^{\op s}}(z)
\]
Let us denote the derivations associated to the three unital spectral triples $(\C A_1,H_1,D_1)$, $(\C A_2,H_2,D_2)$ and $(\C A_1 \ot \C A_2, H, D_1 \ti D_2)$ by $d_1 : \C A_1 \to \B L(H_1)$, $d_2 : \C A_2 \to \B L(H_2)$ and $d : \C A_1 \ot \C A_2 \to \B L(H)$, respectively. According to the proof of Proposition \ref{p:stable} (see \eqref{eq:dhot1}) we need to establish that
\begin{equation}\label{eq:spectralineq}
\| (d_1 \ot 1)_s(z) \| \, , \, \, \| (1 \ot d_2)_s(z) \| \leq \| d_s(z) \| .
\end{equation}

We divide the proof into the four cases depending on the parities of the two unital spectral triples appearing as factors in the external product. 

{\bf Even times even and even times odd:} In these cases the relevant inequality follows since
\[
d_s(z) = (d_1 \ot 1)_s(z) + \big(\ga_1 \ot 1\big)^{\op s} \cd (1 \ot d_2)_s(z),
\]
where we are suppressing the various inclusions into $\B L(H_1 \ot_2 H_2)$. Notice also that the product appearing refers to multiplication of operators in $\B L(H_1 \ot_2 H_2)$. The above identity together with the observation that all values of the derivation $d_1$ anticommute with $\ga_1$ imply that
\[
\begin{split}
(d_1 \ot 1)_s(z) & = \frac{1}{2} d_s(z) - \frac{1}{2} (\ga_1 \ot 1)^{\op s} \cd d_s(z) \cd (\ga_1 \ot 1)^{\op s} \q \T{and} \\
(1 \ot d_2)_s(z) & = \frac{1}{2} (\ga_1 \ot 1)^{\op s} \cd  d_s(z) + \frac{1}{2} d_s(z) \cd (\ga_1 \ot 1)^{\op s} .
\end{split}
\]
We finally use that the $C^*$-norm of $\ga_1 \ot 1 \in \B L(H_1 \ot_2 H_2)$ is equal to one.

{\bf Odd times even:} The proof of the inequality from \eqref{eq:spectralineq} follows the same pattern as the two cases even times even and even times odd. 

{\bf Odd times odd:} Let us consider the Pauli matrices $\si_1 = \ma{cc}{0 & 1 \\ 1 & 0}$ and $\si_2 = \ma{cc}{0 & i \\ -i & 0}$ as selfadjoint unitary operators on the Hilbert space $H = (H_1 \ot_2 H_2)^{\op 2}$. Furthermore, we define the derivations $(d_1 \ot 1)^{\op 2} := \ma{cc}{d_1 \ot 1 & 0 \\ 0 & d_1 \ot 1}$ and $(1 \ot d_2)^{\op 2} := \ma{cc}{1 \ot d_2 & 0 \\ 0 & 1 \ot d_2}$ both with domain and codomain equal to $\C A_1 \ot \C A_2$ and $\B L\big((H_1 \ot_2 H_2)^{\op 2} \big)$, respectively. It then holds that
\[
d_s(z) = \si_1^{\op s} \cd \big(  (d_1 \ot 1)^{\op 2}\big)_s (z) + \si_2^{\op s} \cd \big(  (1 \ot d_2 )^{\op 2}\big)_s (z)
\]
and, using that the Pauli matrices anticommute, we thereby obtain that
\[
\begin{split}
\big(  (d_1 \ot 1)^{\op 2}\big)_s (z) & = \frac{1}{2} d_s(z) + \frac{1}{2} \si_1^{\op s} \cd d_s(z) \cd \si_1^{\op s} \\
\big(  (1 \ot d_2)^{\op 2}\big)_s (z) & = \frac{1}{2} d_s(z) + \frac{1}{2}\si_2^{\op s} \cd d_s(z) \cd \si_2^{\op s} .
\end{split}
\]
Since the $C^*$-norm of the operator $\big(  (d_i \ot 1)^{\op 2}\big)_s (z) \in M_s\big( \B L(H) \big)$ agrees with the $C^*$-norm of $(d_i \ot 1)_s(z) \in M_s\big( \B L(H_1 \ot_2 H_2)\big)$ for $i \in \{1,2\}$ we get the inequality from \eqref{eq:spectralineq} in this final case as well.
\end{proof}


\bibliographystyle{amsalpha-lmp}

\providecommand{\bysame}{\leavevmode\hbox to3em{\hrulefill}\thinspace}
\providecommand{\MR}{\relax\ifhmode\unskip\space\fi MR }
\providecommand{\MRhref}[2]{%
  \href{http://www.ams.org/mathscinet-getitem?mr=#1}{#2}
}
\providecommand{\href}[2]{#2}

\end{document}